\documentclass[final]{amsart}

\usepackage{amssymb}
\usepackage{etoolbox}
\usepackage{hyperref}
\usepackage[mathscr]{euscript}
\usepackage{tikz}
\usepackage{tikz-qtree}
\usetikzlibrary{cd,matrix,arrows,decorations.pathmorphing}
\usepackage{bbm}
\usepackage{tikz-cd}
\usepackage{rotating}
\usepackage{marginnote}
\usepackage[margin=0.7in]{geometry}

\theoremstyle{remark}
\newtheorem{example}{Example}[section]
\newtheorem{remark}[example]{Remark}
\theoremstyle{definition}
\newtheorem{definition}[example]{Definition}
\newtheorem{notation}[example]{Notation}
\newtheorem{goal}[example]{Goal}
\theoremstyle{plain}
\newtheorem{proposition}[example]{Proposition}
\newtheorem{corollary}[example]{Corollary}

\newtheorem{theorem}[example]{Theorem}
\newtheorem*{theorem*}{Theorem}
\newtheorem{lemma}[example]{Lemma}

\newcommand{\DeclareMyOperator}[1]{%
	\expandafter\DeclareMathOperator\csname #1\endcsname{#1}
}

\newcommand{\DeclareMathOperators}{\forcsvlist{\DeclareMyOperator}}

\DeclareMathOperators{Aut,  coker, cd, ch, ed, ind, Hom,  End, Ext, ext,
	Flag,   gr, Mat, Par, rank, rep, rk, Spectra, Schemes, Tors, td,  Vect}

\DeclareMathOperator{\im}{Im}
\DeclareMathOperator{\fil}{filt}

\DeclareMathOperator{\Spec}{Spec}

\DeclareMathOperator{\Mod}{Mod}

\DeclareMathOperator{\trdeg}{trdeg}

\DeclareMathOperator{\rmH}{H}

\newcommand{\shom}{{\mathcal Hom}}

\newcommand{\sC}{{\mathscr C}}
\newcommand{\sE}{{\mathscr E}}
\newcommand{\sF}{{\mathscr F}}
\newcommand{\sG}{{\mathscr G}}
\newcommand{\tsG}{{\mathscr {\tilde G}}}
\newcommand{\sL}{{\mathscr L}}
\newcommand{\sM}{{\mathscr M}}
\newcommand{\sN}{{\mathscr N}}
\newcommand{\sO}{{\mathscr O}}
\newcommand{\sX}{{\mathscr X}}
\newcommand{\sY}{{\mathscr Y}}
\newcommand{\sT}{{\mathscr T}}

\newcommand{\bR}{{\textbf R}}

\newcommand{\fG}{{\mathfrak G}}

\DeclareMathOperator{\Coh}{\mathcal{C}\mathit{\!o\!h}}
\DeclareMathOperator{\Bun}{\mathcal{B}\mathit{\!u\!n}}
\DeclareMathOperator{\Nil}{\mathcal{N}\mathit{\!\!il}}

\newcommand{\sHom}{{\mathcal Hom}}

\newcommand{\sEnd}{{\mathcal End}}

\newcommand{\Fields}{\mathsf{Fields}}
\newcommand{\Sets}{\mathsf{Sets}}
\newcommand{\E}{\mathcal{E}}
\newcommand{\F}{\mathcal{F}}

\newcommand{\bbZ}{\mathbb{Z}}

\newcommand{\Gm}{\mathbb{G}_{\mathrm{m}}}
\newcommand{\QQ}{{\mathbb Q}}
\renewcommand{\AA}{{\mathbb A}}
\newcommand{\ZZ}{{\mathbb Z}}

\newcommand{\T}{\mathrm{T}}
\newcommand{\Y}{\mathrm{Y}}

\newcommand{\G}{\mathrm{G}}

\newcommand{\bn}{{\bf n}}
\newcommand{\bA}{{\bf A}}
\newcommand{\cf}{{\rm CF}}

\newcommand{\fn}{\mathfrak{ n}}

\DeclareMathOperator{\h}{H}

\begin{document}
\title[Essential dimension of Parabolic Bundles]{On the essential ($p$)-dimension of parabolic bundles on curves}

\author[A. Dhillon]{Ajneet Dhillon}
\address{Department of Mathematics, University of Western Ontario, London, Ontario N6A 5B7, Canada}
\email{adhill3@uwo.ca}
\author[D. Valluri]{Dinesh Valluri}
\address{Department of  Computer Science, University of Western Ontario, London, Ontario N6A 5B7, Canada}
\email{dvalluri@uwo.ca}

\begin{abstract}
 We study the essential dimension and essential $p$-dimension of the moduli stack of vector bundles over a smooth orbifold curve containing a rational point. We improve the known bounds
 on this essential dimension and obtain an equality modulo the
 famous conjecture of Colliot-Thelene, Karpenko and Merkurjev. In the case of essential $p$-dimension we obtain an equality.
\end{abstract}
\maketitle
\section{Introduction}

Roughly, the essential dimension of a family of algebraic objects is the number of parameters needed to parameterise a generic family of such objects. This heuristic definition points to its central role it plays in moduli problems. A precise definition can be obtained by observing that there are two ways of defining the dimension of an algebraic variety. Firstly, there is the Krull dimension and secondly one can define dimension as transcendence degree of the function field over the base field. By lifting the second definition to the category of algebraic stacks one arrives at the precise notion of essential dimension. This intriguing invariant is difficult to compute. There is a variation known as the essential $p$-dimension, which is roughly the essential dimension ignoring prime to $p$ information, that is easier to compute. We will recall both of these definitions in section 4 below, see also \cite{reichstein} and \cite{merkurjev}. 

The purpose of this paper is to study these invariants for the moduli stack of vector bundles on an orbifold curve, with coarse moduli of genus at least two. We extend the results of \cite{crelle} in two ways. Firstly, we extend them to essential $p$- dimension. 
This has the virtue of being able to state and prove an equality that is not conjectural, see \ref{t:main}.
Secondly we consider smooth projective curves with an orbitfold structure, i.e. certain kinds of root stacks. The problem is divided into two pieces as in \cite{crelle}.
Given a vector bundle $\sE$ on an orbifold curve, it corresponds to a point of a moduli stack of vector bundles. Hence there is 
a corresponding residual gerbe $\sG(\sE)$ with coarse moduli space $k(\sE)$, the field of moduli of $\sE$.
The essential dimension of $\sE$ breaks down into two component pieces.
 That of understanding the essential dimension of the residual gerbe over the field of moduli and then understanding the transcendence degree of the field of  moduli over the base field. 
 
 The second of these two steps is carried out in section five, roughly amounts to understanding the tangent space to the automorphism group of a parabolic bundle. We present a new approach to this, different to that in\cite{crelle}, using deformation theory and filtered derived categories. This approach is useful in that it has potential to generalise to principal bundles over  groups other than
 the general linear group.
 
 The results of this paper are confined to dimension one, due to the fact that the dimensions of the stacks that we consider can be computed via Euler characteristics. This is no longer the case in higher dimension.  Extensions of Riemann-Roch  to Deligne- Mumford stacks, 
 see \cite{toen} and \cite{edidin}
 play a pivotal role. We recall this theorem in section 3
and compute its terms in the case of an orbifold curve.

The essential dimension of vector bundles on an orbifold curve was first considered in
\cite {nicole}. The results of this paper give a vast improvement over the results in \cite {nicole}.
For example, let's consider a smooth projective curve $X$ with a single orbifold point $x\in X$ point with orbifold structure
$\bbZ/n\bbZ$. Vector bundles  on this curve acquire an action of the group $\bbZ/n\bbZ$
over the orbifold point. Hence by ordering the eigenvalues of the action we obtain a filtration of the orbifold point, that is a parabolic bundle. Let $n_i$ be the dimensions of these vector spaces in this filtration so that $n_0=r$ the rank of the vector bundle being 
considered. Suppose that the vector bundle has degree $d$ and set $h=\gcd(n_i,r,d)$. Let $\Bun^{r,d}_{\bn}$ be the moduli stack
of vector bundles with prescribed data where $\bn=(n_0\ge n_1\ge \ldots \ge n_{e})$. In this paper we show that
\[
\ed(\Bun_{\bn}^{r,d})\le r^2(g-1)+1+ \Flag_\bn + \sum_{p|h} (p^{v_p(h)}-1).
\]
The last sum is over primes dividing $h$. The bound in \cite{nicole} is more difficult to describe, but roughly it is of the form
\[
\ed(\Bun_{\bn}^{r,d})\le r^2(g-1)+1+ \Flag_\bn + F(r).
\]
where the function $F(r)$ is quadratic in the rank $r$, see \cite[12.1]{nicole} for details.

In section 2 of the paper we start by recalling the parabolic-orbifold correspondence. This is an equivalence of categories between vector bundles on a root stack and vector bundles with filtration on its coarse moduli space. The third section is an overview of Riemann-Roch for orbifolds. We perform some calculations that will be useful later. In section 4 we recall essential dimension and its variant 
essential $p$-dimension. We recall the conjecture in \cite{ct} and state its $p$-analogue, see \ref{e:gerbe}. Some results from
\cite{crelle} are recalled and extended to essential $p$-dimension and orbifolds. The fifth section studies the field of moduli of
a parabolic bundle. We use deformation theory methods to understand and bound the transcendence degree of the field of moduli. This is in contrast to the global methods in \cite{crelle}. As stated earlier, this may prove to be useful as these local calculations are more apt to generalisation to other groups. The final section, section six, states and proves our main result \ref{t:main}.

\section*{Acknowledgements} The first named author would like to thank Kirill Zainoulline for suggesting that we consider essential 
$p$-dimension. The first named author thanks NSERC for funding.

\section{The parabolic-orbifold correspondence}

Let $X$ be a scheme and $\sL$ a line bundle on $X$ with section $s\in\rmH^0(X,\sL)$. If $e$ is a positive integer,  we may form the \emph{root stack}
\[
q:X_{\sL,s,e}\rightarrow X,
\]
see \cite{borne}. A lift of an $S$-point, $f:S\rightarrow X$ to the root stack
amounts to a line bundle with section $(\sM,t)$ on $S$ and an isomorphism 
$\alpha:\sM^{\otimes e}\rightarrow f^*\sL$ sending $t^e$ to $s$. The automorphisms are the obvious ones. It follows that there is a universal root line bundle
$\sN$ on $\sX_{\sL,s,e}$ whose $e$th power is the pullback of $\sL$.
We will refer to $e$ as the \emph{ramification index} of the construction.

On the other hand the data $(\sL,s)$ and $e$ determine a notion of \emph{parabolic vector bundle}
on $X$. This is a vector bundle $\sE$ together with a filtration 
\[
\sE_0=\sE\supseteq \sE_1\supseteq \ldots \supseteq \sE_e
\]
and an isomorphism $\sE\otimes \sL^{-1} \cong \sE_e$ such that the composition
\[
\sE\otimes \sL^{-1} \cong \sE_e \hookrightarrow \sE_0\cong \sE\otimes \sO
\]
arises from the section. There is a corresponding category $\Par(\sL,s,e)$ of 
parabolic vector bundles. We refer the reader to \cite{borne} for details.

\begin{theorem}\label{t:correspondence}
	Let $X$ be a noetherian scheme. There is an equivalence of categories
	\[ \Par(\sL,s,e)\cong \Vect(X_{\sL,s,e}) \]
\end{theorem}

\begin{proof}
	See \cite[3.13]{borne}. Let us remark here that the parabolic bundle associated
	to a vector bundle $\sE$ on $\sX$ is obtained by 
	$q_*(\sE\otimes \sN^{-i})=\sE_i$.
\end{proof}

\begin{remark}\label{r:local}
	The root stack admits a nice local description which explains the above
	correspondence quickly. Suppose that $X=\Spec(R)$ is affine and $\sL$ is trivial.
	Then $s\in R$. The scheme
	\[ R[t]/<t^e-s>\]
	has an action of the group scheme $\mu_e$. The quotient stack is the root stack.
	The correspondence comes from the fact that $\mu_e$-equivariant objects are
	just graded objects. For details see \cite{cadman}.
\end{remark}

The root stack construction is easily seen to be functorial in the following
sense, given $f:Y\rightarrow X$ then the root stack $Y_{f^*\sL,e}$ is the
2-pullback of $X_{\sL,e}$ along the morphism $f:Y\rightarrow X$. In other words,
there is a 2-cartesian diagram
\begin{center}
\begin{tikzcd}
Y_{f^*s\sL,e} \arrow[r,"g"] \arrow[d,"q'"] & X_{\sL,e} \arrow[d,"q"] \\
Y \arrow[r,"f"] & X\\
\end{tikzcd}
\end{center}

\begin{proposition}\label{p:flat}
	In the above situation, suppose that $Y\rightarrow X$ is flat. Suppose that
	$\sF$ is a vector bundle on $X_{\sL,e}$ with corresponding parabolic
	vector bundle $\sF_0\supseteq \sF_1\ldots \supseteq \sF_e$. Then 
	the parabolic vector bundle corresponding to $g^*\sF$ is
	$f^*\sF_0\supseteq f^*\sF_1\supseteq \ldots \supseteq f^*\sF_e$.
\end{proposition}

\begin{proof}
	Recall that vector bundle $\sF_i = q_*(\sF\otimes \sN^{-i})$ so that this
	result amounts to essentially flat base change. To make this precise, the
	root stack is locally on $X$ a $\mu_e$ quotient stack, \cite[3.4]{borne}.
	The result now follows from flat base change and the fact that the
	push forward $q_*$ amounts to taking $\mu_e$-invariants of an equivariant sheaf.
\end{proof}

Now assume $X$ is a projective variety over a ground field $k$. We fix Cartier divisors
$D_1,\ D_2,\ldots, D_l$ and positive integers $e_1, e_2,\ldots, e_l$ coprime
to ${\rm char} (k)$.
We let $$\sX = X_{(D_1,e_1),\ldots, (D_l,e_l)}$$
be the corresponding root stack construction.
Corresponding to this there are root line bundles (see \cite{cadman}) written $\sN_i$  
on the root stack $\sX$. We write $q:\sX\rightarrow \Spec(k)$ for
the structure map.

\begin{lemma} \label{l:exact}.
	The morphism $q_*$ is exact where $q$ is 
	the coarse moduli map $q:\sX\rightarrow X$. 
\end{lemma}

\begin{proof}
	This follows from the local description of the root stack see \ref{r:local}.
	Note that the stack is tame.
\end{proof}

\begin{corollary}
	The derived functor $\bR q_*$ preserves the amplitude
	of a bounded complex.
\end{corollary}

\begin{theorem}
	The stack of coherent sheaves on $\sX$, written $\Coh_\sX$ is algebraic.
\end{theorem}

\begin{proof}
	The standard proof in \cite{laumon} can be made to work when combined with the
	following observations. Let $\sF$ be a coherent sheaf on $\sX$. We can find
	integers $n_{ij}$ and vector spaces $H_{ij}$ so that we have epimorphisms
	\[ \sO_X(-n_{ij})\otimes H_{ij} \twoheadrightarrow q_{*}(\sF\otimes\sN_i^{-j}) \]
	where $0\le j\le e_{i}$. By adjointness we obtain a morphism 
	\[ \sN_i^{j}\otimes\sO_\sX(-n_{ij})\otimes  H_{ij}  \rightarrow \sF. \]
	Taking a direct sum of these maps we obtain an epimorphism, this follows
	form the local description, see \ref{r:local} or \cite{vistoli}. Further we can arrange for
	the appropriate higher cohomology to vanish using Serre vanishing. The
	needed presentation comes from considering open subsets of Quot schemes.
	
	For the existence of quot scheme in the current setting, see \cite{starr}.
\end{proof}

We will be interested in the case where $X$ is a smooth projective curve over a field
$k$. Our Cartier divisor will be a closed point $p\in X$ or a finite collection of such points.
 The corresponding
notion of parabolic vector bundle amounts to vector bundle $\sE$ on $X$ and a $k(p)$-point
of a flag variety $\Flag(\sE|_{k(p)},n_1,n_2,\ldots, n_{e-1})$
parameterising subspaces
\[
V_0=\sE_{k(p)} \supseteq V_1\supseteq \ldots \supseteq V_{e-1} \supseteq V_e= \{0\}\]
with 
$\dim_{k(p)} V_i=n_i$. If $\sE$, thought of as a vector bundle on the root stack via
the previous theorem, is allowed to vary in a flat family the numbers $n_i$ along with
$r={\rm rk}(\sE)$ and $d=\deg(\sE)$ do not change.
We will refer to the collection 
\[ (r,d, (p, n_0, n_1,n_2,\ldots, n_{e-1}, n_e))\]
as a parabolic datum. Notice that $n_0=\rk(\sE)$ and $n_e=0$.

We will have occasion to consider many such points
$p_1,\ldots, p_l$ with ramification indices $e_i$ and integers
$$
(n_{0i}=\rk(\sE),n_{i0}\ge n_{1i}\ge n_{2i}\ge \ldots \ge n_{e_ii}=0)=\bn_i.
$$
A moduli stack of parabolic bundles, denoted
\[
\Bun^{r,d}_{\bn,X}
\]
is obtained, here $\bn=(\bn_1,\ldots, \bn_l)$. As the forgetful morphism
\[
\Bun^{r,d}_{\bn,X}\rightarrow \Bun^{r,d}
\]
is represented by Weil restrictions of flag varieties, we obatin an 
alternate proof that the stack is algebraic.

\begin{remark}\label{r:shom}
There is an internal hom object in the category of parabolic vector
bundles. Indeed there is one in the category of vector bundles on
the root stack, hence the assertion follows from the correspondence \ref{t:correspondence}. We would like to describe the parabolic datum
associated to the endomorphism bundle as this will be used later. Let
$\sF$ be a parabolic bundle with datum $(n_0,n_1,\ldots n_{e-1})$ at the
Cartier divisor $D$. Then $\sHom(\sF,\sF)$ has datum
$(m_0,m_1,\ldots ,m_{e-1})$ where
\[
m_d = \sum_{\substack{d\le \lambda < e \\ \lambda=i-j\ {\rm mod}\ e}}(n_i - n_{i+1})(n_j-n_{j+1}).
\]
This can be seen by looking at the $\mu_e$-action on a module of the form
$M\otimes M^\vee$ in the local description, \ref{r:local} and observing that $n_i-n_{i+1}$ is the dimension
of the space where the action has weight $\zeta^i$ for some primitive $e$th root of unity $\zeta$.
\end{remark}

 %%%%%%%%%%%%%%%%%%%%%%%%%%%%%%%%%%%%%%%%%%%

\section{Riemann-Roch for the root stack}\label{s:rr}

\subsection{Riemann-Roch for Deligne-Mumford quotient stacks}  

In this section we recall a version of the Riemann-Roch theorem for Deligne-Mumford stacks, see \ref{t:rreg} below. There are other versions in \cite{toen} and \cite{toen2}. We would like our version of the theorem to
hold in positive characteristic provided our stack is tame. As the theorem in 
\ref{t:rreg} is stated in characteristic zero, some discussion regarding modifications of arguments are needed. We do not need the full strength of this result. We only need to consider the case where our
stack is a quotient stack by a torus. Our discussion centres around a rank one torus, higher rank modifications being left to the reader.

\begin{notation} \label{notation}
	Let $k$ be an algebraically closed field (of arbitrary characteristic) and $Y$ be a smooth $k$-scheme on which $\mathrm{T} = \Gm \rightarrow \Spec(k)$ acts properly, i.e., when the action map $\mathrm{T} \times \Y \rightarrow \Y \times \Y$ is proper. In particular, this implies that the stabilizers are finite.  Assume that when the characteristic of $k$ is non-zero it is coprime to the order of all the stabilizers. 
	Let $\mathrm{N}:= \Hom_{\rm groups}(\mathrm{T},\Gm)$ be the group of characters of $\mathrm{T}$. It is an infinite cyclic group, 
	with generator that we call $\mathrm{t}$, so $\mathrm{N}=<\mathrm{t}>$.
	We can recover $\mathrm{T}$ as $\mathrm{D}(\mathrm{N})$, the diagonalizable group associated to $\mathrm{N}$.  Let $\mathrm{R} = \mathbb{Z}[\mathrm{N}] \otimes_{\mathbb{Z}} \bar{\mathbb{Q}} =  \mathrm{R}(\mathrm{T})\otimes_{\mathbb{Z}} \mathbb{\bar{Q}} = \mathbb{\bar{Q}}[\mathrm{t},\mathrm{t}^{-1}]$ be the ring of representations of $\mathrm{T}$ with coefficients in $\bar{\mathbb{Q}}$.
	We recall from \cite{thom} a construction of Segal: 
	
	For every prime ideal $\mathrm{P}$ of $\mathrm{R}$, we may associate a subgroup $\mathrm{T}_\mathrm{P}$ of $\mathrm{T}$ called the \emph{support} of $\mathrm{P}$. These groups are given by $\mathrm{T}_\mathrm{P} = \mathrm{D}(\mathrm{N}/\mathrm{K}_{\mathrm{P}})$, where 
	$$\mathrm{K}_{\mathrm{P}} := \{ \mathrm{n} \in \mathrm{N} \: / \: 1 - [\mathrm{n}] \in \mathrm{P} \}.$$
\end{notation}

With the notation above we have
\begin{lemma} \label{supp}
	For $\T = \Gm$ the supports are classified as follows: 
	\begin{center}
	$\T_\mathrm{P}=$
	$\begin{cases}
		\mu_l & \text{if} \;  \mathrm{P} = (\mathrm{t}-\zeta_{l}), where \; \zeta_{l} \text{ is a primitive l-th root of unity,} \\
		\Gm & otherwise.
	\end{cases}$
           \end{center}
\end{lemma}

\begin{proof}
	
	When $\mathrm{P} = (\mathrm{t} - \zeta_{l})$ we have $\mathrm{K}_\mathrm{P} = \{ t^{i} \in \mathrm{N} = <\mathrm{t}> \mid / \: 1-\mathrm{t}^{i} \in (\mathrm{t} - \zeta_{l}) \} = <\mathrm{t}^{l}>$. Therefore $\T_\mathrm{P} = \mathrm{D}(\mathrm{N}/\mathrm{K}_{\mathrm{P}}) = \mathrm{D}(<\mathrm{t}>/<\mathrm{t}^{l}>) = \mu_{l}$. When $\mathrm{P}$ is a prime ideal not of the form $(\mathrm{t}- \zeta_l)$ then there is no positive integer $i$ such that $1-\mathrm{t}^{i} \in \mathrm{P}$. Hence $\mathrm{K}_\mathrm{P} = \{1\}$ and $\mathrm{D}(\mathrm{N}/\mathrm{K}_\mathrm{P}) = \mathrm{D}(\mathrm{N}) = \Gm$. 
\end{proof}

The $\T$-equivariant $\G$-theory of a scheme $\Y$, $\G_{0}^{\T}(\Y)$ is a module over the ring of representations $\mathrm{R}(\T)$. Moreover, $\G_{0}^{\T}(\Y) \otimes \bar{\mathbb{Q}}$ is supported at finitely many maximal ideals of $\mathrm{R} = \mathrm{R}(\T)\otimes \bar{\mathbb{Q}}$.

\begin{theorem} [Theorem 5.2 \cite{eg-Duke}]\label{thomloc}
	Let $\Y$ be a $k$-scheme with a $\T$-action as described above. Let $\Y^\mathrm{P}$ be the closed subscheme of fixed points of $\Y$ by $\T_\mathrm{P}$. We have a decomposition of $\T$-equivariant $\G$-theory with coefficeints in $\bar{\mathbb{Q}}$ as follows	
	$$ \G_{0}^{\T}(\Y)\otimes \bar{\mathbb{Q}}  = \bigoplus_{\mathrm{P}} (\G_{0}^{\T}(\Y^{\mathrm{P}})\otimes \bar{\mathbb{Q}})_{(\mathrm{P})}, $$	
	where $\mathrm{P}$ ranges over a finite number of ideals of the form $\mathrm{P} = (\mathrm{t} - \zeta_{l})$ for some primitive $l$-th roots of unity $\zeta_l$, which includes $\zeta_1 = 1$. Here $(\G_{0}^{\T}(\Y^{\mathrm{P}})\otimes \bar{\mathbb{Q}})_{(\mathrm{P})}$ denotes the localization at the prime ideal $\mathrm{P} \subset \mathrm{R}$. In particular $\G_{0}^{\T}(\Y)\otimes \bar{\mathbb{Q}}$ is supported at only finitely many closed points of $\Spec(\mathrm{R})$.
\end{theorem}

\begin{proof}
	The key observation is that there exists an ideal $\mathrm{J} = (\mathrm{t}^{d} - 1) \subset \mathrm{R}$ such that $\mathrm{J}\G_{0}^{\T}(\Y) = 0$, see the proof of \cite[Proposition 5.1]{eg-Duke}. Further, by the Chinese remainder theorem $\mathrm{R}/\mathrm{J} = \bigoplus_{\zeta} \mathrm{R}/(\mathrm{t}-\zeta)$ where $\zeta$ ranges over the $d$-th roots of unity. Theorefore, $ \G_{0}^{\T}(\Y)\otimes \bar{\mathbb{Q}} =  (\G_{0}^{\T}(\Y)\otimes \bar{\mathbb{Q}}) \otimes_{\mathrm{R}} \mathrm{R}/\mathrm{J} = \bigoplus_{\zeta} (\G_{0}^{\T}(\Y) \otimes \bar{\mathbb{Q}})/{(\mathrm{t}-\zeta)} = \bigoplus_{\zeta} (\G_{0}^{\T}(\Y) \otimes \bar{\mathbb{Q}})_{(\mathrm{t}-\zeta)}$. By Thomason's localization theorem  \cite[Theorem 2.1]{thom}, for a prime ideal $\mathrm{P} \subset \mathrm{R}$ the morphism of $\mathrm{R}_{(\mathrm{P})}$-modules induced by the equivariant embedding $i_\mathrm{P} : \Y^{\mathrm{P}} \hookrightarrow \Y$
	$$ (i_\mathrm{P})_{*} : (\G_{0}^{\T}(\Y^{\mathrm{P}})\otimes \bar{\mathbb{Q}})_{(\mathrm{P})} \rightarrow (\G_{0}^{\T}(\Y)\otimes \bar{\mathbb{Q}})_{(\mathrm{P})} $$	
	is an isomorphism. Hence the theorem follows. 
\end{proof}

For an ideal $\mathrm{P} = (\mathrm{t} - \zeta_l)$ assume that $\T_\mathrm{P} = \mu_l$ acts trivially on $\Y$. In such a case, for a $\T$-equivariant sheaf $\sF$ there is a decomposition $\sF = \bigoplus_{\chi \in \hat{\T}_{\mathrm{P}}} \sF_{\chi}$ since the base field is tame. We define an operator $\mathrm{t}_\mathrm{P} : \G_{0}^{\T}(\Y)\otimes \bar{\mathbb{Q}}  \rightarrow \G_{0}^{\T}(\Y)\otimes \bar{\mathbb{Q}} $ associated to a prime ideal $\mathrm{P} = (\mathrm{t} - \zeta_l)$ as follows: 
$$ \mathrm{t}_\mathrm{P}([\sF]) := \bigoplus_{\chi \in \hat{\T}_\mathrm{P}} \zeta_{l}^{k_{\chi}} [\sF_{\chi}], $$
where $k_{\chi}$ is the weight of the character $\chi : \mu_{l} \rightarrow \Gm$. 

\begin{remark}
	When $k = \mathbb{C}$ the notion of $\mathrm{t}_\mathrm{P}$ coincides with the notion of $\mathrm{t}_\mathrm{h}$ in  \cite[Defintion 4.8]{edidin} when $\mathrm{h} = \zeta_l$, identified as a $\mathbb{C}$-point of $\Gm$. Note that in \cite{edidin} if $h$ is in the support of $\G$-theory then it is necessarily an element of finite order in $\mathbb{C}^{*}$ and hence must be a primite $l$-th root of unity for some $l$. 
	
	When $k = \mathbb{C}$ and $\Y = \Spec(k)$, $\G_{0}^{\T}(\Y) = \mathrm{R}(\T)$ is the representation ring of $\T$. In this case the action of $\mathrm{t}_\mathrm{P}$ on a character $\chi$ coincides with the action of $\zeta_{l}^{-1}$ on $\chi$ as defined in section 2.6 of \cite{eg}. 
\end{remark}

\begin{lemma}\label{tpinvariance}
	Recall \ref{notation}. If $\mathrm{P} = (\mathrm{t}-\zeta_{l}) \subset \mathrm{R}$ and $\T_\mathrm{P} = \mu_l$ acts trivially on $Y$ then $(\mathrm{t}_{\mathrm{P}}[{\sF}])^{\T} = [\sF^{\T}] $ 
\end{lemma}
\begin{proof}
	One argues as in  \cite[Lemma 2.8]{eg}.
\end{proof}

We recall some results from \cite{eg-Duke} and \cite{krish}.

\begin{theorem}[\cite{eg-Duke} Theorem 3.1 or \cite{krish} Theorem 4.6]\label{rrkrish}
	Let $p : {\Y}' \rightarrow \Y$ be a finte $\T$-equivariant morphism of schemes such that $\T$ acts properly on $\Y$ (and hence on ${\Y}'$ by proposition 2.1 of \cite{eg}). Then there is a commuting square 
	\[    
	\begin{tikzcd}
		\G^{\T}({\Y}')_{\mathfrak{m}_1} \arrow[r, "\tau_{{\Y}'}^{\T}"] \arrow[d, "p_{*}"]
		& CH_{\T}^{*}({\Y}') \arrow[d, "p_{*}"] \\
		\G^{\T}(\Y)_{\mathfrak{m}_1} \arrow[r,  "\tau_{\Y}^{\T}" ]
		&  CH_{\T}^{*}(\Y)
	\end{tikzcd}
	\]
	such that the horizontal maps are isomorphisms. Here $\mathfrak{m}_{1} = (t-1)$ is the augmentation ideal in $\mathrm{R} = \mathrm{R}(\T) \otimes \bar{\mathbb{Q}}$. 
\end{theorem}
\begin{proof}
	A finite morphism of schemes is proper. So we may apply the functoriality of the equivariant Riemann-Roch morphism $\tau_{\Y'}^{\T}$ to $p$  by \cite{eg-Duke} or \cite{krish}. Moreover,  to get the above commuting diagram one observes that $\hat{\G}^{\T}(\Y') =  \G^{\T}(\Y')_{\mathfrak{m}_1}$, where $\hat{\G}^{\T}(\Y')$ is the completion of $\G^{\T}(\Y')$ at the augmentation ideal $\mathfrak{m}_1$ of $\mathrm{R}$. 
	
	The horizontal maps are isomorphisms due to proposition 2.6 in \cite{eg}.
\end{proof}

\begin{remark}\label{psurj}
	In the above theorem if $p$ is surjective then both the vertical maps are surjective as well. This is \cite[ Lemma 3.5]{eg}.
\end{remark}

Now we will state the Riemann-Roch theorem for geometric quotients with \ref{notation}. Note that this is a slightly generalized version of  \cite[Theorem 3.1]{eg} for $\T=\Gm$. In particular, the following theorem does not assume that the base field $k$ is of characteristic $0$. It does assume that the 
group action is tame, i.e., the characteristic is coprime to the order of all stabilizers. 
We consider the equivariant $\G$-theory and $K$-theory in the following theorem with coefficients in $\bar{\mathbb{Q}}$.

\begin{theorem} \label{rr:anychar}
	Let $\Y$ be a smooth $k$-scheme with a proper $\T = \Gm$-action and  $\Y \rightarrow \mathrm{Z}$ be a geometric quotient. Let $\mathrm{P} \subset \mathrm{R}$ be a prime ideal in the support of $\G_{0}^{\T}(\Y)$, $i_\mathrm{P} : \Y^{\mathrm{P}} \hookrightarrow \Y$ the embedding of the fixed points of $\Y$ by $\T_\mathrm{P}$ and $\mathrm{N}_{\mathrm{P}}$ be the relative normal bundle of $i_\mathrm{P}$. Let $j_\mathrm{P} : \mathrm{Z}^{\mathrm{P}} \hookrightarrow \mathrm{Z}$ be the induced inclusion on the quotients. Then for $\alpha \in \mathrm{K}_{0}^{\T}(\Y)$, we have
	$$\tau_{\mathrm{Z}}(\alpha^{\T}) = \sum_{\mathrm{P} \in Supp(\alpha)} \phi_{\Y} \circ (i_\mathrm{P})_{*} (\frac{\ch^{\T}(\mathrm{t}_\mathrm{P}(i_{\mathrm{P}}^{*}\alpha))}{\ch^{\T}(\mathrm{t}_{\mathrm{P}}(\lambda_{-1}\mathrm{N}_{\mathrm{P}}^{*}))}\td^{\T}(\T_{\Y^{\mathrm{P}}})). $$	
	Here $\phi_{\Y} : CH_{\T}^{*}(\Y) \rightarrow CH^{*}(\mathrm{Z})$ is the isomorphism induced by the geometric quotient $\Y \rightarrow \mathrm{Z}$.
\end{theorem}

\begin{proof}
	The key difference between this theorem and  \cite[Theorem 3.1]{eg} is that we need to replace the elements $\zeta_{l}$ in the support of $\G$-theory with a finite set of ideals of the form $\mathrm{P} = (\mathrm{t} - \zeta_l)\subset \mathrm{R}$.  By making these changes with the help of \ref{supp}, \ref{thomloc} and \ref{tpinvariance}, the theorem verbatim follows the argument in \cite{eg}. We explain the key steps. Step 1 of the proof of  \cite[Theorem 3.1]{eg} is a direct application of \cite[Theorem 3.1(e)]{eg-Duke}. 
Step 2  uses a theorem of Seshadri \cite[Theorem 6.1]{Seshadri} to reduce to the case of step 1. \cite[Theorem 6.1]{Seshadri}  implies that given a geometric quotient $\Y \rightarrow \mathrm{Z}$ by a diagonalizable group $\T$, there exists a finite surjective $\T$-equivariant map $\Y' \rightarrow \Y$ such that $\T$ acts freely on $\Y'$. We remark that  \cite[Theorem 6.1]{Seshadri} is valid in any characteristic.

 In step 3 of \cite[Theorem, 3.1]{eg} we only need to check that when $\beta_\mathrm{P} := \frac{i_{\mathrm{P}}^{*}\alpha}{\lambda_{-1}\mathrm{N}_{\mathrm{P}}^{*}} \in \G_{0}^{\T}(\Y^{\mathrm{P}})_{(\mathrm{P})} $ then $\mathrm{t}_\mathrm{P}(\beta_{\mathrm{P}}) \in \G_{0}^{\T}(\Y^{\mathrm{P}})_{(\mathrm{t}-1)}$. This follows by seeing that $\mathrm{t}_{\mathrm{P}}((\mathrm{t}-\zeta_{l})) = \mathrm{t}_{\mathrm{P}}(\mathrm{t}) - \zeta_{l} = \zeta_{l}(\mathrm{t} - 1)$ and hence $\mathrm{t}_{\mathrm{P}}$ takes $(\mathrm{t}-\zeta_{l})\G_{0}^{\T}(\Y^\mathrm{P})$ to $(\mathrm{t}-1)\G_{0}^{\T}(\Y^\mathrm{P})$.

\end{proof}

Now we recall  the Riemann-Roch theorem for quotient Deligne-Mumford stacks from 
\cite{eg} and \cite{edidin}.

If $\sY$ is algebraic stack we denote its inertia stack by $I\sY$. There is a projection $f:I\sY\rightarrow \sY$. The Euler class of a class $\alpha\in K_0(I\sY)$ will be denoted by $\lambda_{-1}(\alpha)$. 
On a class of a vector bundle $V$ it is given by
$$
\lambda_{-1}([V])= \sum_i (-1)^i [\wedge^i V].
$$
Finally there is a twisting operation
$$
t: K_0(I\sY)\otimes\bar{\QQ} \rightarrow  K_0(I\sY)\otimes\bar{\QQ}
$$
obtained by decomposing a vector bundle into eigenspaces for the inertial action and twisting by the
eigenvalue. For a precise construction we refer the reader to \cite[\S 2.6]{eg} and \cite[\S 4.2]{edidin}.
An example will be computed below.
We denote the normal bundle to $f:I\sY\rightarrow \sY$ by $N_f$.

\begin{theorem}\label{t:rreg}
	Let $\sY$ be a smooth Deligne-Mumford stack with coarse moduli space $\mathrm{Z}$ that is proper over the ground field
	$k$ of tame characteristic. We further assume that $\sY$ is a quotient stack by a torus $\T$, so that $[\Y/\T]=\sY$.
	If $V$ is a vector bundle on $\sY$ then
	$$
	\chi(\sY,V) = \int_{I\sY} \frac{\ch(t f^*V)}{ \ch (t \lambda_{-1}(N_f^*)) } \td(I\sY)
	$$ 
\end{theorem}

\begin{proof}
	Let $\alpha = [V]$ be the class of a vector bundle $V$ on $\sY$ and $\Y \xrightarrow[]{\pi} \mathrm{Z}$ be the geometric quotient. A vector bundle $V$ on $\sY$ is equivalent to a $\T$-equivariant bundle, also denoted by $V$, on $\Y$. Observe that the $\mathrm{K}$-theoretic direct image of $\alpha^{\T} = [(\pi_{*}V)^{\T}]$ is given by $\chi(\mathrm{Z},\alpha^{\T})$. By the Riemann-Roch theorem for the proper map $\mathrm{Z} \rightarrow \Spec{k}$, the pushforward   $\int_{Z} \tau_{\mathrm{Z}}(\alpha^{\T}) \in CH_{*}(\mathrm{pt}) = \mathbb{Q}$ coincides with $\chi(Z, \alpha^{\T})$.  Now we may deduce the required formula using \ref{rr:anychar} by observing that $\chi(\mathrm{Z},\alpha^{\T})= \chi([\Y/\T],\alpha)$, see \cite[Section 4.1]{edidin} and the fact that the inertia stack admits a decomposition $ I\sY = \coprod_{\mathrm{P}} [\Y^{\mathrm{P}}/\T]$. See \cite[4.19]{edidin} and the reference contained within.
\end{proof}

\subsection{Root stacks as quotient stacks}

Let's start by recalling the alternate construction of a root stack in \cite{cadman}. For every postive
natural number $n$ there is a morphism of quotient stacks
$$p_e:[\AA^1/\Gm]\rightarrow [\AA^1/\Gm]$$
given by 
$$z\longmapsto z^e$$
and passing to quotients. The data of a line bundle and section $(\sL, s)$ on $X$ is the same as giving a 
$\Gm$-torsor $E$ on $X$ and an $\Gm$-equivariant morphism $\sigma: E\rightarrow \AA^1$, indeed the section is obtained from
$$X=E/\Gm \rightarrow E\times_{\Gm} \AA^1=\sL.$$
The root stack is constructed as
$$X_{(\sL,s,e)} \cong E\times_{[\AA^1/\Gm], p_e}[\AA^1/\Gm].$$
This realises the root stack as quotient stack by a torus,
$$ [E\times_{\sigma,\AA^1, e} \AA^1/\Gm], $$
where $\Gm$ acts on the right hand $\AA^1$ and the structure map $\AA^1\rightarrow \AA^1$ is raising
to the $e$th power.

In the case where we have multiple line bundles, that  is multiple parabolic points, a similar construction applies and are stack will be a quotient stack by a split torus.

In the situation, where $X$ is a smooth projective curve, the Riemann-Roch theorem, \ref{t:rreg},  applies to the root stack. We will calculate the right hand side of the theorem in this section to obtain a more explicit form.

\subsection{Statement of Riemann-Roch on a root stack}
To set things up we let
$X$ be a smooth projective curve and $p_i$ are closed points on $X$ and let
\[
\sX = X_{((p_1,e_1),\ldots (p_m,e_m))}.
\]
be the root stack.

Given a vector bundle $\sF$ on $\sX$ we  set 
$$
\deg(\sF) = \int_{\sX} c_1(\sF)
$$

\begin{theorem}\label{t:stackRR} Suppose that $k=\bar{k}$ is algebraically closed.
	We preserve the notation above.
	Let $\sF$ be a vector bundle on $\sX$ with parabolic datum
	$(n_{i,0},n_{i,1},\ldots, n_{i,e_i})$ at $p_i$. 
	Then
	\[
	\chi(\sF) =  \deg(\sF) + (1-g)\rk(\sF) - \sum_i 
	\sum_{d=0}^{e_i - 1} \frac{d(n_{i,{d}}-n_{i,{d + 1}})}{e_i}.
	\]	
\end{theorem}

To simplify the notation we will present the proof when there is only one parabolic point, the general result
is a mild modification of the argument presented here. So for the remainder of this section, we assume:
\begin{goal}\label{g:rr}
	Let $X$ be a smooth projective curve over $k=\bar{k}$, an algebraically closed field. Form a root stack $\sX=X_{p,e}$ and let 
	$\sF$ be a vector bundle on $\sX$ with root datum 	$(n_{0},n_{1},\ldots, n_{e})$ 
	at $p$. Then 
	\[
	\chi(\sF) =  \deg(\sF) + (1-g)\rk(\sF) - 
	\sum_{d=0}^{e - 1} \frac{d(n_{{d}}-n_{{d + 1}})}{e}.
	\]	
\end{goal}

The proof will be given below in \ref{ss:together}, after some preliminary calculations.

As a first step towards the proof of this result, we record here a description of the inertia stack
of $\sX$. 
\begin{proposition}
	There is a decomposition
	$$
	I\sX = \sX \amalg \coprod_{\substack{\omega\in\mu_e \\ \omega\ne 1}} B\mu_e.
	$$
	The identifications of the components with $B\mu_e$ can be made so that restriction of the root line bundle $\sN$ to each $B\mu_e$ is the weight one representation of 
	$\mu_e$ on a one dimensional space.
\end{proposition}

\begin{proof}
	A detailed proof can be found in \cite[4.12]{borne}, we sketch the result here.
	An $S$-point of $\sX$ consists of a quadruple $(g,\sM,u,\psi)$ where $g:S\rightarrow X$ is a morphism,
	$\sM$ is a line bundle on $S$, $u$ is a global section of $\sM$ and $\psi:\sM^n\rightarrow g^*\sO(p)$ is
	an isomorphism sending $u^r$ to $g^*(s)$. Here $s$ is the chosen section of $\sO(p)$ vanishing at $p$. Isormophisms of the data are defined in the obvious way. 
	
	If $g^*(s)\ne 0$ then the data is rigid, in other words there are no automorphisms. If this section vanishes, a nontrivial automorphism amounts to multipication by a nontrivial $e$th root of unity. These are the points of the inertia stack. As the group $\mu_e$ is abelian, an automorphism of point of the inertia stack is given by multiplication by an arbitrary $e$th root of unity.
	
	The statement regarding the restriction of the root line bundle $\sN$ can be obtained by observing that 
	$\sM$ is the pullback of $\sN$ and the identification with $\mu_e$ is obtained via identifying
	$\Gm\cong {\rm Aut}(\sM)$.
\end{proof}

Calculation of the integral in \ref{t:rreg} breaks into two cases. The gerbe integral over 
$\coprod_{\substack{\omega\in\mu_e \\ \omega\ne 1}} B\mu_e$ and the integral over $\sX$.

It will be useful to introduce and use the following notation in the calculations.

\begin{notation}\label{n:notation}
	\begin{enumerate}
		\item $\mu_e$ the $e$th roots of unity in $k$.
		\item $\omega$ a primitive $e$th root of unity.
		\item $\chi:\mu_e\rightarrow \Gm$ the inclusion.
		\item $V_\rho$ denotes the vector bundle on $B\mu_e$ corresponding to a representation $\rho$.
	\end{enumerate}
\end{notation}

We begin with the gerbe integral.

\subsection{The gerbe integral}

We will make use of the following facts.
\begin{eqnarray}\label{e:facts}
	{\rm Ch}_*(B\mu_e)\otimes_\ZZ \QQ &\cong& \QQ[t]/et \\
	\int_{B\mu_e} 1 &=& 1/e
\end{eqnarray}

These can be found in the calculations of \cite{edidin} and the references contained therein. The second
statement is easily deduced. 

\begin{lemma}
	Consider the morphism $f:I\sX\rightarrow \sX$ restricted to 
	$$
	f_i:B\mu_e \rightarrow \sX,
	$$
	where $B\mu_e$ is one of the components of the root stack labelled by a nontrivial $n$th root of unity
	$\omega^i$.
	The conormal bundle to the morphism is $V_{\chi}$ (recall \ref{n:notation}).
\end{lemma}

\begin{proof}
	If $G$ is a group scheme acting on $Y$ the inertia subscheme $IY$ is
	$$
	IY= \{(g,x)\in G\times X| gx=x\}\subseteq G\times X,
	$$	
	more canonically it is the equaliser of 
	$$
	G\times X   \rightrightarrows X
	$$
	where one of the morphisms is the projection and the other the action. We start by locally describing the
	inertial scheme in our situation.
	The lemma is local on $X$ so we may pass to $\sO_{X,p}=R$, a discrete valuation ring with parameter $s$.
	Then by the local description, \ref{r:local}, we have that  
	$$
	\sX|_{\sO_{X,p}} \cong [R[t]/<t^e-s>/\mu_e].
	$$
	The coaction of $\mu_e$ is given by
	$$
	R[t]/<t^e-s> \rightarrow R[t,x]<t^e-s, x^e-1>\qquad t\mapsto  xt.
	$$
	Under the decomposition
	$$
	R[t,x]/<t^e-s, x^e-1> \cong \prod_{\omega^i\in\mu_e} R[t]/<t^e-s>,
	$$
	the component of the inertial scheme at $\omega^i$ is the coequaliser of 
	$$
	R[t,s]/<t^e-s> \rightrightarrows R[t,s]/<t^e-s>,
	$$
	where the top arrow is the identity and the other sends $t\mapsto t\omega^i$. The inertial scheme near $\omega^i$, $i\ne 0$, is $R/s$ and its ideal sheaf is $<t>$. The conormal is hence $<t>/<t^2>$ which has
	a weight one action.
\end{proof}

\begin{proposition}
	\label{p:rootintegral}
	Preserving the notation above we let $\omega$ be a primitive $e$th root of unity. Denote by
	$\sN^d|_{\omega^i}$ the $d$th power of the root line bundle restricted to the component of
	$B\mu_e$ labelled by $\omega^i$. Then
	\[
	\int_{} \frac{\ch(tf^*\sN^d|_{\omega^i})}{\ch(t\lambda_{-1} N_f^*) }\td(B\mu_n) = 
	\frac{1}{e}\left( \frac{\omega^{id}}{1-\omega^{-i}}\right).
	\]
\end{proposition}

\begin{proof}
	The line bundle restricts to $\chi^d$ on $B\mu_n$. Hence
	\begin{eqnarray*}
		\frac{\ch(tf^*\sN^d|_{\omega^i})}{\ch(t\lambda_{-1} N_f^*) }
		&=& \frac{\ch(t [V_{\chi^d}])}{\ch(t(1-[V_{\chi^{-1}}]))} \\
		&=& \frac{\ch(\omega^{id}[V_{\chi^d}])}{\ch(1-\omega^{-i}[V_{\chi^{-1}}])} \\
		&=& \left( \frac{\omega^{id}(1+dc_1(\sN))}{1-\omega^{-i}(1-c_1(\sN))}\right)
	\end{eqnarray*}
	Write $t=c_1(\sN)$. This term isn't important as it is torsion, see the facts at the
	start of this subsection. In fact
	all chow groups in non-zero codimension are torsion and do not contribute to the integral. The result follows from the fact that 
	$$
	\int_{B\mu_e} 1 =\frac{1}{e}.
	$$
\end{proof}

\begin{corollary}\label{c:local}
	In the situation of the proposition, we have
	$$
	\int_{\coprod_{\substack{\omega^i\\ i\ne 0}} B\mu_e} \frac{\ch(tf^*\sN^d|_{\omega^i})}{\ch(t\lambda_{-1} N_f^*) }\td(B\mu_e) = \frac{1}{e}\left(\frac{e-1-2d}{2}\right)
	$$
\end{corollary}

\begin{proof}
	We need the following identities pertaining to sums of rational functions in roots of unity:
	\begin{eqnarray*}
		\sum_{i=1}^e \omega^{ik} &=& 
		\begin{cases}
			-1\quad &\text{ for}\quad 0<k<e\\
			e-1\quad &\text{ for}\quad k=0\\
		\end{cases} \\
		\sum_{i=1}^{e-1} \frac{1}{\omega^i-1} &=& 
		-\left(\frac{e-1}{2}\right) \\
		\frac{\omega^{id} -1}{\omega^{i}-1} 
		&=& (\omega^i)^{d-1} +(\omega^i)^{d-2} + \ldots +1 \quad d> 0 \\
		\sum_{i=1}^{e-1} \frac{\omega^{id} -1}{\omega^{i}-1}  &=& e-d\qquad
		0<d<e \\
		\sum_{i=1}^{e-1} \frac{\omega^{id} }{\omega^{i}-1}  &=&
		\frac{e-2d+1}{2}\qquad
		0<d\le n
	\end{eqnarray*}
	The first is well known. You could for example, sum it as a geometric series.
	The second follows from aplying $d\log$ to 
	$$
	X^{e-1}+\cdots + 1 =\prod_{i=1}^{e-1} (X-\omega^i).
	$$
	The third is an easy division. The fourth follows from the third and second by interchanging sums.
	The final identity is an easy consequence of the previous ones.
	
	The required integral is 
	\begin{eqnarray*}
		\frac{1}{e}\sum_{i=1}^e \frac{\omega^{id}}{1-\omega^{-i}} &=&
		\frac{1}{e} \left( \frac{\omega^{i(d+1)}}{\omega^i-1}\right) \\
		&=& \frac{1}{e}\left( \frac{e-2(d+1)+1}{2}\right) \\
		&=& \frac{1}{e}\left( \frac{e-1-2d}{2}\right)
	\end{eqnarray*}
\end{proof}

The tool for leveraging this calculation to a general vector bundle is the following proposition.
\begin{proposition}\label{p:local}
	Let $\sF$ be a vector bundle on $\sX$ with parabolic datum $(r,p, (n_0, n_1,\ldots , n_e))$. 
	Then
	$$ f^*\sF|_{B\mu_e,\omega^j} = \bigoplus_{d=0}^{e-1} V_{\chi^i}^{(n_d-n_{d+1})}.$$
	where $i\ne 0$.
\end{proposition}

\begin{proof}
	By \cite[3.12]{borne} there is a Zariski neighbourhood $U$ of $p$ so that $\sF|_{\pi^{-1}(U)}$ is a direct sum of 
	line bundles. Using \cite[3.1.2]{cadman}, we may assume that $\sF|_{\pi^{-1}(U)}$ is a direct sum of 
	powers of root line bundles. We are now reduced to the case that $\sF=\sN^i$. By the identification 
	of components of the inertia stack with $B\mu_e$, we have that $\sN|_{B\mu_e, \omega^j}$ is
	isomorphic to $[V_{\chi^l}]$ where $l\equiv i\mod e$.
	Now, the parabolic line bundle corresponding to $\sN^i$ is given by
	$$
	\frac{a}{e} \longmapsto q_*(\sN^{i-a}) \cong \sO^{\lfloor \frac{i-a}{e} \rfloor},
	$$
	see \cite[3.11]{borne}. Recall that $q$ is the coarse moduli map $q:\sX\rightarrow X$.
	We write $i=\alpha e+l$. Then the number $\lfloor \frac{i-a}{e} \rfloor$ jumps one 
	as $a$ passes from $l$ to $l+1$ so that $n_l-n_{l+1}=1$.
\end{proof}

\subsection{The integral over $\sX$}

The starting point for this calculation is a comparison of tangent sheaves of $X$ and $\sX$. Once again
the calculation is entirely local at $p$, the stacks being isomorphic away from $p$. If $R=\sO_{X,p}$ then
recall
\[
\sX|_{\Spec(R)} = [\Spec(R[t]/<t^e-s>)/\mu_e]. 
\]
The pullback map on cotangent bundles is 
\[
Rds\rightarrow R[t]dt/<t^e-s>\qquad ds\mapsto t^{e-1}dt.
\]
Let $G=\coker(T_X\rightarrow T_{\sX})$. Dualising the above calculation we see that $G$ is the coherent sheaf
on $\sX$ given by $R[t]/t^{e-1}$ with $\mu_e$ action given by $t\mapsto \omega t$. The sheaf can be globally
resolved by line bundles on $\sX$ as
$$
0\rightarrow \sN^{1-e} \rightarrow \sO_{\sX} \rightarrow G \rightarrow 0.
$$ 
As Todd classes are multiplicative we find that 
$$
\td(\sX) = \td(X)(1+\frac{1-e}{2}c_1(\sN)).
$$

\begin{proposition}\label{p:global}
	Let $\sF$ be a vector bundle on $\sX$ of rank $r$. We have
	$$
	\int_{\sX} \ch(\sF)\td(\sX) = \int_{\sX} c_1(\sF) + r(1-g) + \frac{r(1-e)}{2e},
	$$
	where $g$ is the genus of $X$.
\end{proposition}

\begin{proof}
	There is a coarse map $\pi:\sX\rightarrow X$.
	We have
	\begin{eqnarray*}
		\int_{\sX} \ch(\sF)\td(\sX) &=& \int_{\sX} (r+c_1(\sF))(1+c_1(\pi^*(T_X))/2)(1-\frac{e-1}{2}c_1(\sN)) \\
		&=& \int_\sX c_1(\sF) + r\int_X c_1(T_X) + \frac{r(1-e)}{2}\int_\sX c_1(\sN) \\
		&=& \int_{\sX} c_1(\sF) + r(1-g) + \frac{r(1-e)}{2e}.
	\end{eqnarray*}
	Notice that $\int_\sX c_1(\sN) = \frac{1}{e}\int_\sX \pi^*\sO(p)$ by the defining property of 
	the root line bundle. The last line, hence is by the projection formula.
\end{proof}

\subsection{Putting it all together}\label{ss:together}

We will write $\int_{\sX} c_1(\sF) = \deg(\sF)$.

\begin{proof} of \ref{t:stackRR}.
	Recall we need to show that 
	\[
	\chi(\sF) =  \deg(\sF) + (1-g)\rk(\sF) - 
	\sum_{d=0}^{e - 1} \frac{d(n_{{d}}-n_{{d + 1}})}{e}.
	\]	
	This is now just a matter of putting the above calculations together.
	By \ref{c:local} and \ref{p:local}
	We have that 
	$$
	\int_{\coprod_{\substack{\omega^i\\ i\ne 0}} B\mu_e} \frac{\ch(tf^*\sF^d|_{\omega^i})}{\ch(t\lambda_{-1} N_f^*) }\td(B\mu_e) = \frac{r(e-1)}{2e} - \sum_{d=0}^{e-1}\frac{d(n_d-n_{d+1})}{e}.
	$$
	The result follows from \ref{p:global}.
\end{proof}

\begin{corollary}\label{c:general}
	Suppose that $k$ is not algebraically closed and the $p_i$ are closed points
	of $X$. Then 
	$$
	\chi(\sF) =  \deg(\sF) + (1-g)\rk(\sF) - \sum_i \deg(p_i)
	\sum_{d=0}^{e_i -1} \frac{d(n_{i,d}-n_{i,d+1})}{e_i}.
	$$
\end{corollary}

\begin{proof}
	One can base change to an algebraically closed field. Note that in the
	diagram
	
	\begin{center}
		\begin{tikzcd}
			\sX_{\bar{k}} \arrow[d, "\bar{q}"] \arrow[r,"g"] &  \sX \arrow[d,"q"] \\
			X_{\bar{k}}  \arrow[r,"f"]  & X
		\end{tikzcd}
	\end{center}
	the functors $f^*,\ g^*, q_*, \bar{q}_*$ are all exact so that flat base
	change applies, see \ref{p:flat}.
\end{proof}

Let $\sF$ be a parabolic vector bundle on $\sX$ with parabolic datum
as specified earlier given by
$$ \bn_i=(n_{i0}=\rk(\sF)\ge n_{i1}\ge \ldots\ge n_{ie_i}=0).$$
at the points $p_i$ as previously specified.

Recall that in  (\ref{r:shom}) we described the parabolic datum on
$\sHom(\sF,\sF)$. Lets write down  the Euler characteristic of this
bundle under the hypothesis $\bar{k}=k$. To simplify the statement it will be
helpful to introduce the notation
$
\Flag_{\bn_i}(V)
$
for the flag variety parameterising sequences of subspaces
$$
V_0\supseteq V_1\supseteq \ldots V_{e_i}=0
$$
of a fixed vector space
$V$ with $\dim V=\dim V_0 =\rk(\sF)$ such that $\dim V_i=n_i$.

\begin{proposition}\label{p:shom}
	In the above setting, in particular $k=\bar{k}$,  we have 
	$$ \chi(\shom(\sF,\sF)) = (1-g)\rk(\sF)^2 -
	\sum_{i=1}^l \dim\Flag_{\bn_i}(\sF|_{p_i}).$$
\end{proposition}

\begin{proof}
	We have $\deg(\shom(\sF, \sF)) = 0$ by the  the splitting principle, and $\rk(\shom(\sF,\sF)) = \rk(\sF)^{2}$. To simplify notation we may assume that $\sF$ is ramified only at a single point $p$ with parabolic data $\bf{n}$, the general result follows by an easy induction. By \ref{r:shom} and \ref{c:general} it is enough to prove that 
	$$\sum_{d = 0}^{e-1} \frac{d(m_{d} - m_{d+1})}{e} = \dim \Flag_{\bn}(\rk(\sF|_{p})),$$ where $m_d $ is defined in remark \ref{r:shom}.
	To simplify notation, we write $\delta_i=n_i-n_{i+1}$. We get
	\begin{eqnarray*}
		\sum_{d = 0}^{e-1} \frac{d(m_{d} - m_{d+1})}{e} &=& \frac{1}{e}\sum_{d=1}^{e-1} d\sum_{\substack{i-j=d \\ \mod e}} \delta_i\delta_j  	\\
		&=& \frac{1}{2e}\sum_{d=1}^{e-1} \left( d\sum_{\substack{i-j=d \\ \mod e}} \delta_i\delta_j  + 
		(e-d)\sum_{\substack{i-j=-d \\ \mod e}} \delta_i\delta_j\right)	\\
		&=& \frac{1}{2} \left( \sum_{d=1}^{e-1} \sum_{\substack{i-j=d \\ \mod e}} \delta_i\delta_j \right) \\
		&=& \frac{1}{2}\sum_{i\ne j}\delta_i\delta_j \\
		&=& \sum_{0\le i < j\le e-1} (n_i - n_{i+1})(n_j-n_{j+1})
	\end{eqnarray*}	
	
	The dimension of the flag variety is $$\dim\Flag_{\bn}(\rk(\sF|_{p})) = \sum_{i=1}^{e-1} n_i(n_{i-1}-n_i).$$ One checks that these
	two formulas agree, recall $n_e=0$.
\end{proof}

If $\sE_{1}$ and $\sE_{2}$ are coherent sheaves over $\sX_{K}$, for a field $K \supset k$, we denote
$$\chi(\sE_{2}, \sE_{1}) := \dim_{K}\Hom(\sE_{2}, \sE_{1}) - \dim_{K} \Ext(\sE_{2},\sE_{1})$$.

When $\sE_{1}$ and $\sE_{2}$ are vector bundles, $\chi(\sE_{2}, \sE_{1})$ coincides with $\chi(\sHom(\sE_{2}, \sE_{1}))$.

\begin{remark} \label{torsion} 
	Recall, \ref{r:local}, that locally that the root stack is a quotient of $R[X]/<X^e-s>$
	by $\mu_e$. The ring $R[X]/<X^e-s>$ is $\bbZ/e$-graded. An equivariant module
	over this module amounts to a graded module.
	
	In our present situation of a root stack over a smooth curve, we see that
	a coherent sheaf $\sE$ over $\sX_{K}$ can be written as a direct sum $\sF \oplus \sT$, where $\sF$ is a vector bundle and $\sT$ a torsion sheaf supported at finitely many points. We say that $\bn$ is the parabolic datum of $\sE$ if $\bn$ is the parabolic datum of $\sF$, in particular $\rk(\sE|_{p}) := \rk(\sF|_{p})$ for a point $p$ in $\sX_{K}$. 
\end{remark}

\begin{lemma} \label{alt} With the notation in \ref{torsion}, we have
	$ \chi(\sF, \sT) = -\chi(\sT, \sF)$ and $\chi(\sT,\sT) = 0$.
\end{lemma}
\begin{proof}
	The question is local on the base so we may assume that our curve
	is $\Spec(R)$ where $R$ is a DVR. Let $t$ be a uniformizing parameter of $R$.
	For modules over a DVR the result is true by elementary calculations.
	The result for the root stack follows from the previous remark by
	passage to graded modules.
\end{proof}

\begin{lemma}\label{l:coherent}For a coherent sheaf $\sE$ over $\sX_{K}$, 
	$$ \chi(\sE,\sE) = (1-g)r^{2} - \sum_{i=1}^{l}  \dim\Flag_{\bn_i}( \sE|_{p_i})  $$
\end{lemma}

\begin{proof}
	From \ref{torsion} it follows that $\chi(\sE,\sE) = \chi(\sF,\sF) + \chi(\sF, \sT) + \chi(\sT,\sF) + \chi(\sT,\sT)$. Now the lemma follows from \ref{p:shom} and \ref{alt}.
\end{proof}

%%%%%%%%%%%%%%%%%%%%%%%%%%%%%%%%%%%%%%%%%%%%%

\section{Essential dimension}

Let $\Fields/k$ be the category of field extensions of $k$. Consider a functor
$$F:\Fields\rightarrow\Sets.$$ Given a field extension $L/k$ and $x\in F(L)$ we say that
a subextension $K\subseteq L$ is a \emph{field of defintion of $x$}, and write 
$x\rightsquigarrow K$ if there is an $x'\in F(K)$ with $F(i)(x')=x$ where $i:K\hookrightarrow L$. We define the essential dimension of $x$ by
\[
\ed x = \inf_{x\rightsquigarrow K} \trdeg_k K,
\]
where the infimum is over all possible fields of definition. The \emph{essential dimension of} $F$ is defined to be
\[ \ed F = \sup_{\substack{ x\in F(L) \\ L\in\Fields_k}} \ed x \]

Let $p$ be a prime number.
We will consider the following variation obtained by throwing away prime to $p$ data, see \cite{merkurjev}.
In the above situtation, we say that $K$ is a \emph{$p$-field of definition of $x$} and write
$x\rightsquigarrow_p K$ if there are inclusions in $\Fields/k$
\begin{center}
	\begin{tikzcd}
	K \ar[r,hook] & K' \\  & L \ar[u,hookrightarrow] \
	\end{tikzcd}
\end{center}
with $K'/L$ a finite extension of degree prime to $p$ and $x'\in F(K)$ so that $x'$ and $x$ have the same image in $F(K')$.
The \emph{essential $p$-dimension of $x$} is then defined by
\[
\ed_p x = \inf_{x\rightsquigarrow_p K} \trdeg_k K,
\]
where the infimum is over all possible $p$-fields of definition. Finally, the \emph{essential $p$-dimension of} $F$ is defined to be
\[ \ed_p F = \sup_{\substack{ x\in F(L) \\ L\in\Fields_k}} \ed_p x \]

An algebraic stack produces such a functor $F$ by considering isomorphism classes of objects,
the essential dimension of which we refer to as the essential dimension of the stack.

\begin{example}\label{e:gerbe}
	Consider a $\Gm$-gerbe $\sG$ over a field $k$ of index $n=p_1^{a_1} \ldots p_\alpha^{a_\alpha}$ with $p_i$ prime. Then we have
	\[ \ed \sG\le \sum_{i=1}^\alpha (p_i^{a_i}-1), \]
	and this is conjecturally an equality. It is known to be an equality when $\alpha=1$ or $n=6$.  See \cite{ct}, 
	\cite{merkurjev} and \cite{involution}.
	The situation for the essential $p$-dimension is simpler,
	\[ \ed_p\sG = v_p(\ind\sG)-1.\]
	This is easily reduced to the previous case by remarking that prime to $p$-torsion in the Brauer group can be removed
	by passing to a prime to $p$ extension.
\end{example}
 
In this article we will be concerned with the essential ($p$-) dimension of $\Bun^{r,d}_{\bn,X}$ where $X$ is a smooth projective curve and the parabolic points are closed points.
We will recall some theorems from \cite{crelle} that will be useful in our context.

For now, lets work in a slightly more general context. Let $X$ be a projective 
scheme with a $k$-point and choose a collection $D_1$, $D_2$, \ldots $D_l$ of effective Cartier divisors and some positive integers $e_1,\ldots , e_l$
and form the corresponding root stack $q:\sX\rightarrow X$.

Consider a vector  bundle $\sF$ on the root stack defined over some field $l$ containing $k$.
Let $\fG(\sF)$ be the residual gerbe in $\Bun^{r,d}_{\sX}$ of a parabolic bundle $\sF$. The coarse moduli
space of this gerbe, is called the field of moduli of $\sF$ and is denoted $k(\sF)$.
There is a finite extension $L/k(\sF)$ so that $\fG(L)\ne\varnothing$, so 
we may find a parabolic vector bundle $\sF'$ that is a form of $\sF$ that is defined
over $L$. Following \cite{crelle}, we consider
$$
A:=\End(p_*\sF')
$$
where $p:\sX_L\rightarrow \sX_{k(\sF)}$ is the projection. This is just the algebra
of the ordinary vector bundle underlying $\sF'$ which preserves the parabolic 
structure.

One of the main results (stated for projective schemes) of \cite{crelle} is:

\begin{theorem}\label{t:main2}
	In the above situation, consider a field extension $K\supseteq k(\sF)$.
	Set $d=[L:k(\sF)]$.
	There is an equivalence of categories between 
the category of projective $A\otimes_{k(\sF)} K$-modules of rank $1/d$
and the groupoid $\fG(\sF)_K$.
\end{theorem}

\begin{proof}
	As stated above, this is \cite[5.3]{crelle}, and we assert that the proof
	goes through in our more general context of root stacks. We describe here
	quickly the functors in each direction. To produce a module from a point of
	$\fG(F)_K$, say $\sE$ consider the module 
	$$ M = \Hom(p_*(\sF)\otimes K, \sE).$$
	To see that $M$ is projective of the correct rank, consider a field $L$ 
	containing $l$ and $K$, so that after base change to $L$, $\sE$ and $\sF$
	are isomorphic. Observe that
	\begin{eqnarray*}
M \otimes_K L &\cong &	 \Hom(p_*(\sF)\otimes K, \sE) \otimes_K L \\
&\cong & p_*\sHom(p_*(\sF)\otimes K, \sE) \otimes_K L \\
&\cong & p_{L,*} \sHom(p_*(\sF)\otimes L, \sE\otimes L)  \\
&\cong &  \Hom(p_*(\sF)\otimes L, \sF\otimes L),  \\
	\end{eqnarray*}
using \ref{p:flat}. It follows that $M$ is projective of the correct rank.

In the opposite direction, given a module $M$ over $A_L$, one considers the
sheaf 
$$\pi_*\sF_L\otimes_{A_L} M.$$
As per the argument in \cite[5.3]{crelle} these functors give the required equivalence.
\end{proof}

Armed with this result, we are reduced to studying the essential dimension  of the functor of projective modules over
a finite dimensional algebra. We recall here some pertinent definitions and results from \cite{crelle}. All proofs and further
details can be found there. We fix for now a finite dimensional (noncommutative) $k$-algebra $A$. Let $j(A)$ be its Jacobson radical.
Given a nonnegative rational number $r$, we denote by $\Mod_{A,r}$ the category of projective modules over $A$ of rank $r$. Recall that $r$ is defined by $dr=m$ where $P$ is a projective
module with $P^d=A^m$. The functor
\[
\Mod_{A,r}:\Fields_k\rightarrow \Sets
\]
that sends a field to isomorphism classes of projective $A\otimes_k K$-modules of rank $r$,
is a determination functor, that is
\[
\Mod_{A,r}=\begin{cases}
\{*\} & \text{a singleton} \\ \emptyset. & \end{cases}
\]

\begin{proposition}\label{p:summary}
	\begin{enumerate}
		\item If $\fn$ is a nilpotent two-sided ideal of $A$. Then $\Mod_{A,r} = \Mod_{A/\fn,r}$.
		\item If $A\cong B_1\times B_2$ then \[\Mod_{B_1,r}\times \Mod_{B_2,r}\cong \Mod_{A,r}.\]
		\item For coprime integers $n$ and $d$ we have
		\[\Mod_{A,1/d}\cong \Mod_{A,n/d}.\]
	\end{enumerate}
\end{proposition}	

\begin{proof}
	See \cite[3.2,3.3,3.5]{crelle}.
\end{proof}

\begin{proposition}\label{p:primesplit}
	Let $l/k$ be a finite prime to $p$ extension. Consider the functors
	\begin{eqnarray*}
		\Mod_{A,r}:\Fields/k\rightarrow \Sets \\
		\Mod_{A_l,r}:\Fields/l\rightarrow \Sets. \\
	\end{eqnarray*}
The $\ed_p(\Mod_{A,r})=\ed_p(\Mod_{A_l,r})$.
\end{proposition}

\begin{proof}
	The inequality $\ed_p(\Mod_{A,r})\ge \ed_p(\Mod_{A_l,r})$ is clear. Take $M\in \Mod_{A,r}(L)$. If\begin{center}
		\begin{tikzcd}
		K \ar[r,hook] & K' \\  & L \ar[u,hookrightarrow] \
		\end{tikzcd}
	\end{center}
is a $p$-field of definition for $M$ then 
\begin{center}
	\begin{tikzcd}
	lK \ar[r,hook] & lK' \\  & lL \ar[u,hookrightarrow] \
	\end{tikzcd}
\end{center}
is a $p$-field of definition for $M\otimes_L lL$ where $lL$ is a compositum. Further, 
\[\trdeg_l lK = \trdeg_k K. \]
\end{proof}

\begin{proposition}\label{p:prime} Fix a prime $p$.
	Let $\sE$ be an indecomposable vector bundle on $\sX_K$ and suppose $\sX$ has a $k$-point. Then 
	\[
	\dim_k \End(\sE)/j(\E))\le \rk(\sE)\quad \text{and}\quad
	v_p(\dim_k \End(E)/j(E))\le v_p(\rk(E)),
	\]	
	where $j(\sE)$ is the Jacobson radical of $\End(\sE)$.
\end{proposition}

\begin{proof}
	This is \cite[Lemma 4.2]{crelle} in the first case. The second case is similar, so let's recall the proof.
The ring	$\End(\sE)/j(\End(\sE)$ is a division ring and the fiber over a rational point of $X$ is a module over it. The result follows from the fact that every module over a division ring is free.
\end{proof}

Finally we need the following result from \cite{crelle}.

\begin{proposition}\label{p:ed-dimension}
	Let $A$ be division ring with centre $k$. Then
	$$
	\ed_p(\Mod_{A,r}) \le \ed (\Mod_{A,r}) < r\dim_k (A).
	$$
\end{proposition}

\begin{proof}
	The assertion about essential $p$-dimension is trivial. The non-trivial inequality is
	by \cite[3.7]{crelle}.
\end{proof}

\begin{proposition}\label{p:rank} 
Suppose that $\sE$ is vector bundle on $\sX_K$ of rank $r$. 
Then
\[
\ed_{k(\sE)}(\fG(\sE)) \le r-1 \qquad\ed_{k(\sE),p}(\fG(\sE)) \le v_p(r)-1.
\]
Recall that $\fG(\sE))$ is the residual gerbe of $\sE$ in the stack of bundles.
\end{proposition}

\begin{proof}
The first assertion is proved in \cite[5.5]{crelle} so we concentrate on the second. By standard arguments, see
 loc. cit. we can
assume that $K/k(\sE)$ is a finite extension of degree $d$. We write $\pi:\sX_K\rightarrow \sX_{k(E)}$ for the projection. As in loc. cit. we
can decompose 
\[
\pi_* \sE \cong \bigoplus_i \sE_i^{n_i}
\]
where $\sE_i$ is indecomposable and $\End(\sE_i)/j(\sE_i) \cong D_i$ for some division ring. Further we have a decomposition
\[
\End(\pi_* \sE)/j(\sE) \cong \prod_i M_{n_i\times n_i} (D_i),
\]
where $M_{n\times n}(A)$ is the ring of $n\times n$ matrices over a ring $A$.

Each of the division rings decomposes as
\[
D_i\cong \bigotimes_l D_{i,l},
\]
where $D_{i,l}$ has index $l$, and the tensor product is over prime divisors of the index of $D$.
Note that by the theory of division rings we can split each of the factors $D_{i,l'}$ for $l\ne l'$ 
by passing to a prime to $l$ extension of the ground field.

As the dimension of $D_{i,l}$ is a power of $l$,  we have by \ref{p:prime}, $\dim_{k(\sE)} D_{i,l} = v_l(\dim_{k(E)} D_i) \le v_l(\rk(E_i))$.

We have
\begin{eqnarray*}
	\ed_p(\Mod_{\End(\pi_*\sE),1/d)}) &= &\sum \ed_p(\Mod_{M_{n_i\times n_i}(D_i)}, 1/d) \quad by\ \ref{p:summary} \\
	&= & \sum \ed_p({\Mod_{D_i},n_i/d}) \quad  \\
	&= &\sum\ed_p(\Mod_{D_{i,p},n_i/d}) \quad by\ \ref{p:primesplit} \\
	&< & \sum\frac{n_i}{d} v_p(\dim_{k(E)} D_i) \quad \ref{p:ed-dimension}\\
	&\le &\sum \frac{n_i}{d} v_p(\rk(\sE_i))\quad by\ \ref{p:prime}.
\end{eqnarray*}
The result follows from \ref{t:main2}.
\end{proof}

\section{The field of moduli of a parabolic bundle}

In this section $X$ will be a smooth projective curve over $k$.
We form the root stack
$$\sX = X_{(p_1,e_1),\ldots, (p_l,e_l)}.$$  The points $p_i$ are distinct.
Corresponding to this there are root line bundles (see \cite{cadman}) written $\sN_i$  
on the root stack $\sX$ and coarse moduli map
$$q:\sX\rightarrow X.$$

\begin{remark}\label{r:local2}
	In this situation we can find an open affine cover $V_i=\Spec(R_i)$ 
	of $X$ so that each open contains at most one orbifold point. When there is an orbifold point
	in $R_i$, of ramification index $e$ say, the open set $U_i=q^{-1}(V_i)$ has the following description. We can  assume that
	the line bundles are trivial on $V_i$ and that $s\in R_i$ is the section that vanishes at the orbifold point.
	Then the scheme 
	\[ \tilde R_i := R_i[X]/<X^e-s>\]
	has an action of the group scheme $\mu_e$. The quotient stack is the root stack.
    For details see \cite{cadman}. It follows that vector bundles on $U_i$ are just projective
    modules over $\tilde R_i$ with a grading by $\Hom_{\mathrm{groups}}(\mu_e,\Gm)$. By further shrinking the open set
    they become free modules with grading. 
    
    When the section does not vanish, or there is no orbifold point, the action of $\mu_e$ is free and the quotient is the open set
    $V_i=U_i$. In this case vector bundles with action by $\mu_e$ amount to vector bundles on $V_i$.
     In particular, when the orbifold point is removed from the open set, i.e over the
    open set $U_i\setminus {p}$, this observation applies.
\end{remark}

\begin{definition}\label{d:essentiallyFree}
	We will call a vector bundle $\sF$ on a root stack $\sY$ \emph{essentially free} if
	we can find a presentation for $\sY$ as desribed in the remark, ie
	$$
	 \sY= [  R[X]/<X^e-s>/\mu_e],
	$$
	and $\sF$ corresponds to a free module on $R[X]/<X^e-s>$ with $\mu_e$-action.
\end{definition}

If $k\rightarrow A$ is a $k$-algebra we denote by $X_A,\ \sX_A$ etc the base change to $A$.

\subsection{A review of deformation theory}

Consider a vector bundle on $\sF$ on $\sX_A$ where $k\rightarrow A$ is a local Artinian $k$-algebra.
Consider a square zero extension of local Artinian $k$-algebras
$$
0\rightarrow I \rightarrow B \rightarrow A\rightarrow 0.
$$
Given an open set  $V\subset X$
we will often abuse notation and write $V$ when we really mean $q^{-1}(V)$.

\begin{proposition}
	In the above situation there is an affine cover $U_i, 1\le i \le N,$ of $X$ so that 
	\begin{enumerate}
		\item there is at most one orbifold point in each $U_i$,
		\item the vector bundle $\sF_{U_i}$ is essentially free,
		\item there is a lift of $\sF_{U_i}$ to a vector bundle $\tilde\sF_i$ on $U_i\subseteq \sX_B$
		which is essentially free,
		\item given a homomorphism $\rho: \sF\rightarrow \sF$ there is a local lift to 
		$$ \tilde\rho: \tilde\sF_i\rightarrow\tilde\sF_i.$$
		\item any two lifts of the homomorphism differ by a section of
		$ \Gamma(U_i,\shom(\sF,\sF\otimes_A I))$.
	\end{enumerate}
\end{proposition}

\begin{proof}
 The first two items follow from \ref{r:local2} or \cite[3.12]{borne}. 
 For the third and fourth we are in the situation of \ref{r:local2}.  Our base changed stacks are
 $$ [\Spec(R_B[t]/<t^{e}-x>)/\mu_e]\hookleftarrow [\Spec(R_A[t]/<t^{e}-x>)/\mu_e]=q^{-1}(U_i)_A.$$
 Our vector bundle on $[\Spec(R_A[t]/<t^{e}-x>)/\mu_e]$ amounts to a direct sum
 of copies of $R_A[t]/<t^{e}-x>$ with $\mu_e$ action. The action
can be chosen to act on each direct summand individually and  is tantamount to a $\Hom_{\mathrm{groups}}(\mu_e,\Gm)$-grading.
 This lifts in the obvious way to $R_B[t]/<t^{e}-x>$ along with an endomorphism.
 The fifth item is a standard diagram chase.
\end{proof}

Consider an endomorphism $\theta : \sF\rightarrow \sF$. There is an associated morphism
$$
[\theta,-]=\theta_*-\theta^* :\sEnd(\sF)\rightarrow \sEnd(\sF)
$$
Let $P(\sF,\theta)$ be the cone of $\theta_*-\theta^*$, shifted by one, it is a complex concentrated in degrees 0 and 1. One should view this complex as the dual of the cone of $[\theta,-]$.

\begin{theorem}\label{t:usualdef} In the above situation:
	\begin{enumerate}
		\item There is an obstruction in $\h^2(\sX,P(\sF,\theta)\otimes_A I)$ whose vanishing is necessary
		and sufficient for a lift of $(\sF,\theta)$ to $\sX_B$.
		\item When the obstruction vanishes, the space of lifts is an affine space  abstractly isormophic to
		$\h^1(\sX,P(\sF,\theta)\otimes_A I)$.
		\item The automorphism of a lift is 
		$\h^0(\sX,P(\sF,\theta)\otimes_A I)$.
	\end{enumerate}
\end{theorem}

\begin{proof}
	We use Cech cohomology on a cover as in the previous proposition. A more canonical proof, using the cotangent complex can be found in \cite[Ch. IV]{illusie}.
	
	We write $U_{ij}=U_i\cap U_j$
	and use analogous notation for higher intersections.
	Choose an isomorphism $g_{ij}:\tilde\sF_i \rightarrow \tilde\sF_j$ over $U_{ij}$ so that the following diagram commutes:
	
	\begin{equation}\label{eq:lift}
		\begin{tikzcd}
			0 \arrow[r] & \sF|_{U_{ij}}\otimes_A I \arrow[r] \arrow[equal]{d} &\tilde \sF_i|_{U_{ij}} \arrow[r] \arrow["g_{ij}"]{d} & \sF|_{U_{ij}}  \arrow[r]  \arrow[equal]{d} & 0 \\
			0 \arrow[r] & \sF|_{U_{ij}} \otimes_A I\arrow[r] & \tilde \sF_j|_{U_{ij}} \arrow[r] & \sF|_{U_{ij}} \arrow[r] & 0. \\
		\end{tikzcd}
	\end{equation}
	To simplify notation we will write $\sF_i$ for $\sF|_{U_i}$.
	
	On triple overlaps we have
	$$
	g_{ik}^{-1}g_{jk}g_{ij}:\tilde \sF_i\rightarrow \sF_i.
	$$
	Now let $1$ be the identity map of $\sF_i$. Using the diagram (\ref{eq:lift}), one finds that,
	$$
	1-g_{ik}^{-1}g_{jk}g_{ij} = c_{ijk} :\sF_i\rightarrow \sF_i,
	$$
	for some $c_{ijk}$ sections of $\sEnd(\sF)\otimes_A I$ over the triple intersection.
	Similary, on double overlaps we have
	$$
	\tau_{ij}:g_{ij}\tilde \theta_i g_{ij}^{-1} - \tilde \theta_j :\sF_j\rightarrow \sF_j \otimes_A I.
	$$
	
	We claim that $(c_{ijk},\tau_{ij})$ is a 2-cocycle. The Cech complex for $P(\sF,\theta)$ looks like:
	\begin{center}
		\begin{tikzcd}
			C^0(U_i,\sEnd(\sF)) \arrow{r}& 	C^1(U_i,\sEnd(\sF)) ) \arrow["\partial"]{r}  &	C^2(U_i,\sEnd(\sF)) \\
			 C^0(U_i,\sEnd(\sF)) \arrow{r} \arrow{u} & 	C^1(U_i,\sEnd(\sF)) ) \arrow["\partial"]{r} \arrow{u} &	C^2(U_i,\sEnd(\sF)) \arrow{u}{[\theta,-]} \\
		\end{tikzcd}
	\end{center}
	To check this we need to show
	\begin{eqnarray*}
		\partial(c_{ijk}) &=& 0 \\
		\partial(\tau_{ij}) &=& [\theta, c_{ijk}] 
	\end{eqnarray*}
To check the first item observe that
\begin{eqnarray*}
	1 &=& (g_{ik}^{-1} g_{jk}g_{ij})(g_{ij}^{-1} g_{jk}^{-1} g_{kl}^{-1} g_{jl} g_{ij} )
	(g_{ij}^{-1}g_{jl}^{-1}g_{il})(g_{il}^{-1}g_{kl}g_{ik})\\
	&=& (1-c_{ijk})(1+g_{ij}^{-1} c_{jkl} g_{ij} )(1+c_{ijl})(1-c_{ikl})\\
	&=& 1 + \partial(c_{ijk}).
\end{eqnarray*}
In the above, to pass from the second to last line, notice that $g_{ij}^{-1}c_{jkl}g_{ij}=c_{jkl}$ by
\ref{eq:lift}. To check the second condition,
\begin{eqnarray*}
	\partial(\tau_{ijk}) &=& \tau_{ij} - \tau_{ik} + \tau_{jk} \\
	&=& g_{ij} \tilde \theta_i g_{ij}^{-1} - \tilde \theta_j  - g_{ik}\tilde \theta_i g_{ik}^{-1}
	+\tilde \theta_k + g_{jk}\tilde \theta_k g_{jk}^{-1} - \tilde \theta_k \\
    &=& g_{jk}g_{ij} \tilde \theta_i g_{ij}^{-1}g_{jk}^{-1} - g_{jk}\tilde \theta_j g_{jk}^{-1} 
    -g_{ik}\tilde \theta_i g_{ij}^{-1} + g_{jk}\tilde \theta_k g_{jk}^{-1} \\
    &=& g_{ik}^{-1}g_{jk}g_{ij}\tilde \theta_i g_{ij}^{-1}g_{jk}^{-1}g_{ik} - \tilde \theta_i \\
    &=& -[\theta, c_{ijk}].
\end{eqnarray*}
We have used the diagram (\ref{eq:lift}) multiple times in the above.
	
	If this cycle is exact, by say $(b_{ij}, \tau_i)$
	then we can alter the gluing data and lifts to
	$$
	g_{ij} + b_{ij}\quad\text{and}\quad \tilde\theta_i+\tau_i
	$$
	so that they glue globally as needed.
    The verification that this works is a repeat of the calculations above.
    
    For part (3), the identification is obtained by observing that if $\alpha$ is an automorphism of
    a lift that preserves the lifted endomorphism, then $1-\alpha$ is a required global section of the 
    complex. For part (2), any two lifts are locally the same, as all modules involved are essentially free.
    One then reduces (2) to (3) to obtain the required cocycle. Details are omitted.
\end{proof}

\subsection{A review of the filtered derived category}\label{ss:filtered}

We will need to make use of the filtered derived category of an abelian category $\bA$. The theory is spelled out in detail in \cite[Ch. V]{illusie} but let us spend a few paragraphs recalling some of its main points
that will be needed. 

Given an abelian category $\bA$, we denote by $\cf(\bA)$ the category of chain complexes in $\bA$ equipped with a finite descending filtration. Given an object $C^\bullet$ of $\cf(\bA)$ we will often denote its filtration 
by 
\[
F^n(C^\bullet)\supseteq F^{n+1}(C^\bullet)\supseteq \ldots F^m(C^\bullet)=0.
\]
The morphisms of the category are chain maps respecting the filtration. There is a functor denoted $\gr$
from $\cf(\bA)$ to the category of graded complexes. Given an object $M$ of $\cf(\bA)$ and an integer
$n$, we can shift its filtration and obtain a new object $M(n)$ of $\cf(\bA)$ with
$$
F^i(M(n)) = F^{i+n}(M(n)), \quad\text{ and an unchanged underlying complex.}
$$
A filtered quasi-isomorphism is a morphism that is a quasi-isomorphim on all pieces of the filtration.
The filtered derived category, $D_{\fil}(\bA)$ is obtained from $\cf(\bA)$ by inverting filtered quasi-isomorphisms.
This allows to define the filtered extension groups by
$$
\Ext_{\fil}^n(L,M) := \Hom_{D_{\fil}(\bA)}(L,M[n]),
$$
see \cite[ChV. 1.2.3]{illusie}.

Consider two filtered complexes, $L$ and $M$. We can forget the filtration and form the usual chain
complex $\Hom^\bullet(L,M)$ whose $k$th piece is
$$
\Hom^k(L,M) =\prod_i \Hom(L^i, M^{i+k}).
$$
The differential has the usual sign rule. There is a subcomplex, $\Hom^\bullet_{\cf(\bA)}(L,M)$ of 
$\Hom^\bullet(L,M)$ defined by
$$
\Hom^k_{\cf(\bA)}(L,M) =\prod_i \Hom_{\fil(\bA)}(L^i, M^{i+k}),
$$
i.e those homomorphisms that preserve the filtration.
This allows us to equip $\Hom^\bullet(L,M)$ with a filtration, given by
$$
F^n\Hom^\bullet(L,M) := \Hom^k_{\cf(\bA)}(L,M(n)). 
$$
One checks that there is a chain map
$$
\gr\Hom^\bullet(L,M)\rightarrow \Hom^\bullet(\gr L,\gr M).
$$
A filtered complex $I^\bullet$ is said to be \emph{filtered injective} if 
${F^n}I^{\bullet}$ is a complex of injectives for every $n$. If the underlying abelian category has
enough injectives then $\cf(\bA)$ has enough filtered injectives, \cite[Ch. V. 1.4.4]{illusie}.
We will make use of 
\begin{lemma}\label{l:gr}
	If $M$ is filtered injective then 
	$$
	\gr\Hom^\bullet(L,M)\rightarrow \Hom^\bullet(\gr L,\gr M).
	$$
	is an isomorphism.
\end{lemma}

\begin{proof}
	See \cite[Ch. V 1.4.1]{illusie}.
\end{proof}

Depending on how one develops the theory, one proves (or defines)
$$
R\Hom^\bullet(L,M) := R\Hom^\bullet(L,I)
$$
where $M\rightarrow I$ is a quasi-isomorphism and $I$ is filtered injective.

\begin{proposition}
	One has 
	$$
	\Ext_{\fil}^n(L,M) = H^n(F^0R\Hom^\bullet(L,M)).
	$$
\end{proposition}

\begin{proof}
	See \cite[Ch. V, 1.4.6]{illusie}.
\end{proof}

\begin{remark}\label{r:nilfil}
One way of obtaining a filtered sheaf is by starting with a sheaf  $\sE $ equipped with a nilpotent  morphism  $\theta  $. We define  $F ^ i \sE =\im (\theta ^ i)  $.
With this filtration we obtain a filtered morphism

$$\theta: \sE \longrightarrow \sE (1). $$
\end{remark}

\subsection{Vector bundles with nilpotent endomorphism}

Let $\Nil_{n,\sX}$ denote the stack of vector bundles with nilpotent endomorphism on $\sX$ defined as follows. 
For a $k$-scheme $S$, 
\begin{itemize} \item Objects of $\Nil_{n,\sX}(S)$ are pairs $(\sG , \theta)$ where $\sG$ is a vector bundle on $\sX_S$ and $\theta$ is a nilpotent endomorphism of $\sG$. Each of the sheaves $\coker \theta^i$ are assumed to be flat over $S$. Further $\theta^n=0$.
	\item A morphism between $(\sG, \theta)$ and $(\sG', \theta')$ is an isomorphism of sheaves $\alpha : \sG \rightarrow \sG'$ with $\alpha  \theta = \theta'  \alpha$
\end{itemize}

Note that this stack is algebraic as the forgetful functor $\Nil_{\sX,n}\rightarrow \Bun_{\sX}$ is 
representable, see \cite[02ZY]{stacks-project}.

Our goal in this section is to prove that $\Nil_{n,\sX}$ is a smooth stack and find its dimension at a given $K$-point  $(E,\phi)$ for a field $K \supset k$. We give a proof that is different to the one in \cite {crelle}. The proof that we give is based upon deformation theory arguments rather than the global construction in \cite{crelle}. Part of the needed deformation theory has been stated in \cite{laumon88} using the cotangent
complex. We will work things out from scratch so that we do not need to develop the theory of the filtered 
cotangent complex.

Given a $k$-point $(\sG_0,\theta)$ of the stack $\Nil_{\sX,n}$ there is an induced descending filtration by the images of
$\theta$, that is $F^i\sG_0 = \im(\theta^i)$ so that $F^n\sG_0=0$. The sheaves $F^i\sG_0$ are all locally free
as $\sX$ is smooth and of dimension 1.

\begin{lemma}
	Consider a square zero extension of rings
	$$
	0\rightarrow I \rightarrow \tilde R \rightarrow R\rightarrow 0
	$$
	and a split surjection of $R$-modules
	$$
	\begin{tikzcd}
		M \arrow[twoheadrightarrow]{r}{p} & N \arrow[bend right]{l}[swap]{\sigma}.
	\end{tikzcd}
    $$
    Suppose that we have a lift of this diagram to a diagram of $\tilde R$-modules 
    $$
    \begin{tikzcd}
    	\tilde M \arrow[twoheadrightarrow]{r}{\tilde p} & \tilde  N \arrow[bend right]{l}[swap]{\tilde \sigma}.
    \end{tikzcd}
    $$
    Then the lift $\tilde \sigma$ can be chosen so that it splits $\tilde p$.
\end{lemma}

\begin{proof}
	We have a diagram
	$$
	\begin{tikzcd}
	0\arrow{r} & I\otimes  M \arrow{r} \arrow{d} & \tilde M \arrow{r} \arrow{d} & M\arrow{r} \arrow{d}& 0 \\
	0\arrow{r} & I\otimes  N \arrow{r} \arrow[shift left=4]{u} & \tilde N \arrow{r} \arrow[shift left=4]{u}{\tilde \sigma}  & N\arrow{r} \arrow[shift left=4]{u}{\sigma} & 0. 
	\end{tikzcd}
	$$
	We have that $\tilde p \circ \tilde \sigma - 1_{\tilde N}$ gives a homomorphism $N\rightarrow I\otimes N$
	as both of these maps lift $1_N$. Using the spliting this can be extended to a morphism
	$f:N\rightarrow I\otimes M$. The morphism $\tilde\sigma - f$ lifts $\sigma$ and is a splitting.
\end{proof}

\begin{proposition}
	Recall our running root stack construction $q:\sX\rightarrow X$ over a curve.
	Conisder a square zero extension of Artinian local $k$-algebras,
	$$ 
	0\rightarrow I \rightarrow B\rightarrow A\rightarrow 0.
	$$
	Consider a lift $(\sG,\theta)$ of $(\sG_0,\theta_0)$ to an $A$-point of $\Nil_{\sX,n}$. Recall that
	there is an induced filtration on $\sG$ by the images of the powers of $\theta$.
	Then there is an open cover $V_i$ of $X$ so that on $U_i:=q^{-1}(V_i)$ we have that
	\begin{enumerate}
		\item $V_i$ contains at most one orbifold point
		\item all the sheaves $F^i\sG$ are essentially free as in \ref{r:local2}, and if the
		there is no orbifold point in $U_i$ they are free
		\item there are lifts of $F^i\tsG_j$ to modules over $O_{\sX_B}|_{U_i}$ of the form described in 
		\ref{r:local2}
		so that the quotients $F^i\tsG_j/F^{i+1}\tsG$ are all flat over $B$
		\item there are lifts $\tilde \theta_i$ of $\theta|_{U_i}$ to each $F^0\tsG_i$ so that 
		$F^j\tsG_i = \im \tilde \theta_i^j$.
		\item there are isomorphims 
		$$g_{ij}: F^0\tsG_i|_{U_{ij}} \rightarrow F^0\tsG_j|_{U_{ij}}$$
		preserving the filtrations (but not $\theta$!) so that the following diagrams commute:
		$$
		\begin{tikzcd}
			0 \arrow{r} & I \otimes F^0\sG|_{U_{ij}} \arrow[equal]{d} \arrow{r }& F^0\tsG_i|_{U_{ij}} \arrow{r} \arrow{d}{g_{ij}}& F^0\sG|_{U_{ij}} \arrow{r} \arrow[equal]{d} & 0 \\
			0 \arrow{r} & I \otimes F^0\sG|_{U_{ij}} \arrow{r} & F^0\tsG_g|_{U_{ij}} \arrow{r} & F^0\sG|_{U_{ij}} \arrow{r}  & 0. \\	
		\end{tikzcd}
		$$
		\end{enumerate}
\end{proposition}

\begin{proof}
	The first two assertions are straightforward. The third would be easy except it is not perfectly clear how to ensure flatness. To do this we induct on the filtration. Suppose that we have lifted 
	$F^{i}\sG_j$ to $F^{i}\tilde\sG_j$. Then we have a diagram
	$$
	\begin{tikzcd}
		  &                                &                              &  F^{i+1}\sG_j \arrow[dashed]{dl} \arrow{d} \\
		0\arrow{r} & I\otimes F^i\sG_j/F^{i+1}\sG_j \arrow{r} & F^i\tilde \sG_j/F^{i+1}\sG_j \arrow{r} & F^i\sG_j \arrow{r}& 0 \\
	\end{tikzcd}
	$$
	The dotted arrow exists as $F^{i+1}\sG_j$ is of the form described in \ref{r:local2}. 
	The lift is construced by taking its preimage under the map $F^i\sG_j\rightarrow F^i\sG_j/F^{i+1}\sG_j$.
	The local criteria for flatness ensures the flatness of the quotient. The details are worked out on
	\cite[page 127]{lepotier}.

    To construct the local lift of $\theta$ we can proceed by a descending induction on the filtration.
    It is clear how to lift $\theta$ to $F^{n-1}\tsG_j$. For the induction step, choose an arbitrary
    lift of $\theta$ to a surjection
    $\tau: F^{n-i}\tsG_j\rightarrow F^{n-i+1}\tsG_j $. Any lift of $\theta$ is surjective by Nakayama and
    as all the modules are, surjections will be split. In particular, the vertical arrows in the diagram
    $$
    \begin{tikzcd}
    	F^{n-i}\tsG_j \arrow{r} \arrow[two heads]{d}{\tau}   & \F^{n-i}\sG|_{U_j} \arrow[two heads]{d} \\
    	F^{n-i+1}\tsG_j \arrow{r}                 & \F^{n-i+1}\sG|_{U_j}.  \\   	
    \end{tikzcd}
    $$
    By the previous lemma the splittings can be chosen in a compatible way.
    We have an inclusion $ F^{n-i+1}\tsG_j \subseteq F^{n-1}\tsG_j$ which gives two
    lifts of $\theta$ to $ F^{n-i+1}\tsG_j$. First is the lift from the inclusion and $\tau$ and there
    is the previously constructed lift in the inductive step.
     These differ by a morphism
    $$
    \lambda: F^{n-i+1}\tsG_j\rightarrow F^{n-i+1}\tsG_j\otimes I.
    $$
    The morphism can be extended to $F^{n-i}\tsG_j$ via the splitting. We keep the notation $\lambda$ for
    the extension. The required inductive step is obtained by considering $\tau+\lambda$.
    
    To obtain the isormophisms, once again proceed by descending induction on the filtration. Notice
    that we may extend as we have split surjections $F^{i}\twoheadrightarrow F^{i+1}$ and the kernels 
    are free of the same rank on each open set.
\end{proof}

Let $(\sE_0,\theta_0)$ be a $k$-point of $\Nil_{n,\sX}$. We view $\sE_0$ as a sheaf filtered by the images
of $\theta$, that is 
$F^i\sE_0 = \im(\theta_0^i)$. Consider the filtered complex
\[
P(\sX,\sE_0,\theta_0)_{\fil} := [ \sEnd(\sE_0)\stackrel{[\theta_0,-]}{\rightarrow} \sHom_{\fil}(\sE_0,\sE_0(1)) ],
\]
concentrated in degrees 0 and 1. Some words are in order regarding the differential. We view 
$\theta_0$ as a filtered morphisms $\sE_0\rightarrow \sE_0(1)$ so that post composition with $\theta_0$ is
obviously defined. Note that $\theta_0$ also gives a morphism $\sE_0(-1)\rightarrow \sE_0$ so that if 
$\phi$ is a section of the degree 0 sheaf then $\phi\circ \theta_0 :\sE_0(-1)\rightarrow \sE_0$. 
We can shift filtration again to obtain the required morphism. In summary,
$$
[\theta_0,\phi ]:= \theta_0\phi - \phi\theta_0 ,
$$
suitably interpreted.

\begin{theorem}
	We preserve the notation above. Consider a square zero extension 
	$$
	0\rightarrow I \rightarrow B\rightarrow A \rightarrow 0
	$$
	of artinian local $k$-algebras. Suppose we have a lift of $(\sE_0,\theta_0)$ to an $A$-point
	$(\sE,\theta)$ of $\Nil_{n,\sX}$. Then:
	\begin{enumerate}
		\item There is an obstruction to lifting  	$(\sE,\theta)$
	to a $B$-point of $\Nil_{n,\sX}$ in $H^2(\sX, F^0(P\otimes_k I))$. 
		\item If the obstruction vanishes, the space of lifts is a torsor under
		$H^1(\sX, F^0(P\otimes_k I))$.
		\item The automorphism group of a lift is identified with
		$H^0(\sX, F^0(P\otimes_k I))$.
	\end{enumerate}
\end{theorem}

\begin{proof}
Most of the work has already been done in the prior proposition. One just needs to modify \ref{t:usualdef}
using the prior proposition. This is now straightforward.
\end{proof}

\begin{theorem}
	The obstruction constructed above vanishes so that the stack $\Nil_{n,\sX}$ is smooth.
\end{theorem}	

\begin{proof}
	We have an long exact sequence
	$$
	\rightarrow H^1(\sX, F^0\sHom_{\fil}(\sE,\sE(1))\otimes I)\rightarrow H^2(\sX, F^0(P\otimes I))
	\rightarrow H^2(\sX, F^0\sEnd_{\fil}(\sE)\otimes I)\rightarrow .
	$$
	The last term vanishes so it suffices to show that the map 
	$$
	[\theta, -]:H^1(\sX, F^0\sEnd_{\fil}(\sE)\otimes I)\rightarrow H^1(\sX, F^0\sHom_{\fil}(\sE,\sE(1))\otimes I)
	$$
	is surjective. Now, as $\sE$ is locally free and hence dualisable,  one has 
	$$
	H^1(\sX, F^0(\sEnd_{\fil}(\sE))) \cong H^1(\sX, F^0 R\sHom(\sE,\sE))
	$$
	and a similar result holds for the other cohomology group. Now to prove surjectivity, by the five lemma, it
	suffices to show that 
	$$
	H^1(\sX, \gr^i R\sEnd_{\fil}(\sE)\otimes I)\rightarrow H^1(\sX, \gr^i R\sHom_{\fil}(\sE,\sE(1))\otimes I)
	$$
	for $i>0$.
	By \cite[Ch. V 1.4.8.1]{illusie} and the discussion in \ref{ss:filtered}, we have 
	$$H^1(\sX, \gr R\sEnd_{\fil}(\sE)\otimes I) \cong \Ext^1(\gr \sE, \gr\sE )
	\quad\text{and}\quad
	H^1(\sX, \gr R\sHom_{\fil}(\sE,\sE(1))\otimes I) \cong \Ext^1(\gr \sE, \gr\sE )
	$$
	The $\alpha$th associated graded pieces of the morphism $[\theta,-]$ look like
	$$
	\begin{tikzcd}
		\bigoplus_i\Ext^1(\gr^i \sE, \gr^{i+\alpha}\sE) \arrow{r}{\theta^*} \arrow{d}{\theta_*} & \bigoplus_i\Ext^1(\gr^{i-1}\sE , \gr^{i+\alpha}\sE) \\
		\bigoplus_i \Ext^1(\gr^i \sE, \gr^{i+\alpha+1}\sE). & 
	\end{tikzcd}
	$$
	with $\bigoplus_i \Ext^1(\gr^i \sE, \gr^{i+\alpha+1}\sE) = \gr^\alpha R\sHom_{\fil}(\sE,sE(1))$.
	Given a class $(x_i)\in \bigoplus_i \Ext^1(\gr^i \sE, \gr^{i+\alpha+1}\sE)$ we proceed by
	 induction on $i$ to construct a preimage under $[\theta,-]$ of this class.
	Let us call our constructed preimage to be $(y_i)$ with $y_i\in \Ext^1(\gr^i \sE, \gr^{i+\alpha}\sE)$.
	We can take $y_0=0$ and then $y_1$ can be chosen as 
	$$
	\theta: \gr^{i-1}\sE \twoheadrightarrow \gr^i\sE.
	$$ 
	Suppose that we have chosen $y_0,\ldots , y_k$. To construct the next extension class, consider the
	difference $x_{k+1}-\theta_*(y_k)$ and argue as above.
	
\end{proof}

\begin{theorem}\label{t:dimension}
	The stack $\Nil_{n,\sX}$ is smooth over $k$. Its dimension at the $K$-valued point given by a coherent sheaf $\sE$ on $\sX_{K}$ and $\theta \in \End(\sE)$ with $\theta^{n} = 0$ is 
	
	$$ \dim_{(\sE,\theta)}(\Nil_{n, \sX}) = (g-1) \sum_{i=1}^{n} r_{i}^{2} + \sum_{i=1}^{n} \sum_{j } \dim_{K} \Flag_{\bn_{j}^{(i)}}\left( \frac{\im\theta^{i-1}}{\im\theta^{i}}|_{p_j} \right) ,
	$$
	where $r_{i}$ denotes the rank of the  coherent sheaf $\im(\theta^{i-1})/\im(\theta^{i})$  and $\bn_{j}^{(i)}$ the parabolic datum of $\frac{\im\theta^{i-1}}{\im\theta^{i}}$ at $p_{j}$.
	
	For parabolic structures on coherent sheaves, potentially with torsion, recall remark \ref{torsion}.
\end{theorem}

\begin{proof}
	By the above, we have that the dimension is given by 
\begin{eqnarray*}
	\chi(F^0(P(\sE,\theta))) &=& \chi (F^0\sEnd(\sE)) - \chi(F^0\sHom(\sE,\sE(1)))\\
	&=& \sum_{i\ge 0} \dim_k \Hom(\gr^i \sE, \gr^i\sE) - \dim_k
	\Ext^1(\gr^i \sE, \gr^i\sE). \\
\end{eqnarray*}
 The result follows from
	\ref{p:shom} and \ref{l:coherent}.
\end{proof}

\begin{lemma}
	Let $\sC$ be the closure of a  point by $E$ in $\Bun^{r,d}_{\bn}$, then 
	\begin{align*}
	\dim_{k} {\sC} = \trdeg_{k}(k(E)) - \dim_{K} \End(E)
	\end{align*}
\end{lemma}

\begin{proof}
	The stack $\Bun^{r,d}_{\bn}$ is locally a quotient stack as it is for ordinary vector bundles. One way to 
	prove this is to observe that the map forgetting the parabolic structure is representable in flag varieities.
	Hence,
	we may assume $\sC$ to be a quotient stack $[U/H]$ for some scheme $U$ and some algebraic group $H$ . Let $\sG \hookrightarrow \sC$ be the residual gerbe. We have the following Cartesian square

	\begin{center}
		\begin{tikzpicture}
		\matrix (m) [matrix of math nodes,row sep=2em,column sep=4em,minimum width=2em] {
			R & U \\
			\sG  & \sC \\ 
		};
		\path[->] (m-1-1) edge (m-1-2);
		\path[->] (m-1-1) edge node[auto]{$H$}(m-2-1);
		\path[->] (m-2-1) edge (m-2-2);
		\path[->] (m-1-2) edge node[auto]{$H$}(m-2-2);
		
		\end{tikzpicture}
	\end{center}

	We have $\dim_{k}(U) - \dim_{k}(\sC) = \dim_{k}H$ and $\dim_{k(\sG)}R - \dim_{k(\sG)}\sG = \dim_{k}H$. Combining them with the equations $\dim_{k}R = \dim_{k}U$ (since $R$ is an open dense subscheme of $U$) and $\dim_{k}U = \trdeg_{k}k(U) = \trdeg_{k}(k(\sG)) + \dim_{k(\sG)}U$ we get the required formula.

\end{proof}

A parabolic vector bundle is said to be \emph{indecomposable} if it
cannot be written as a non-trivial direct sum of parabolic vector bundles.

\begin{remark}
	The category of parabolic vector bundles satisfies the bichain conditions
	in \cite{atiyah}. In particular by Lemma 6 in \emph{loc. cit.}, every 
	endomorphism of an indecomposable bundle is either nilpotent or an
	automorphism.
\end{remark}

\begin{corollary} \label{trdeg}
	Assume that $K$ is algebraically closed and that 
	 $\sE$ is an indecomposable vector bundle with parabolic datum $\bn$ on $\sX_{K}$.
	 Take $\phi$ to be a general element of the Jacobson radical $j(\sE)$.
	  Let $r_{i}$ be the rank of $\im(\phi^{i-1})/\im(\phi^{i})$.   Then 
	\begin{equation*}
	\trdeg_{k}k(\sE) \leq 1 + (g-1) \sum_{i} r_{i}^2 + \sum_{j} \dim_{K}\Flag_{\bn_{j}}(\sE|_{p_{j}})
	\end{equation*}
\end{corollary}
\begin{proof}
	 By the above remark we have
	  $\End_{\sX}(\sE)/j(\sE) = K$, notice that $K$ is algebraically closed.
	
	Let $\mathcal{C}$ be the closure of a point given by ${\sE}$ in $\Coh_{\sX_{K}}$. By the previous lemma,
	$$
	\dim_{k}\mathcal{C} = \trdeg_{k}k({\sE}) - \dim_{K}\End_{\sX}({\sE}) 
$$	
	We can find a natural number so that $(j(\sE))^{n} = 0$.
	 Let $\mathcal{N} \subset \Nil_{n, \sX}$ be the closure of the points $(\sE,\phi)$ with $\phi \in j(\sE)$ such that each of the sheaves $\im(\phi^{i-1})/\im(\phi^{i})$ has rank $r_{i}$. 
	
	There is a forgetful morphism $\mathcal{N} \rightarrow \mathcal{C}$ whose generic fiber 
	 an open dense subscheme of $j(\sE)$.  So we have,	
	$$
	\dim_{k}{\mathcal{N}} \geq \dim_{k}{\mathcal{C}} + \dim_{K}{j(\sE)} 
	= \dim_{k}{\mathcal{C}} + \dim_{K}\End_{\sX}(\sE) -1 
	=   \trdeg_{k}k({\sE}) - 1
	$$
	From the previous theorem, \ref{t:dimension}, and \cite[11.1]{nicole} we have, 
	$$ \trdeg_{k}k({\sE}) \leq 1 + (g-1) \sum_{i=1}^{n} r_{i}^{2} + \sum_{j} \dim_{K}\Flag_{\bn_{j}}(\sE|_{p_{j}})
	$$
	
\end{proof}

\begin{remark}
	We have that $k(E) = k(E \otimes_{K}L)$ for any field $L \supset K$.
\end{remark}

\begin{lemma}\label{trdeg2}
	We assume that $g(X)\ge2$.
	Let $\sE$ be a vector bundle of rank $r$, degree $d$ and parabolic datum $\bn$ over $\sX_{K}$. If $\sE$ is not simple,
	in other words $\sE$ has an endomorphism that is not a scalar,
	  then
	$$
	\trdeg_{k}(k(\sE)) \leq (g-1)(r^{2} - r) + 2 + \sum_{i} \dim_{K}\Flag_{\bn_{i}}(\sE|_{p_{i}})
	$$
\end{lemma}

\begin{proof}
	By the above remark we may assume $K$ is algebraically closed. Hence by Krull-Schmidt $\sE$ can be written as a direct sum of indecomposable vector bundles $\sE_{\alpha}$  over $\sX_{K}$ of rank $r_{\alpha} \geq 1$ and parabolic data $\bn^{(\alpha)} = (\bn_{1}^{(\alpha)} \ldots \bn_{l}^{(\alpha)})$, where $\bn_{j}^{(\alpha)}$ is the parabolic datum at the point $p_{j}$ . The above corollary says that 
	$$
	\trdeg_{k}k(\sE_{\alpha}) \leq 1 + (g-1) \sum_{i} r_{i \alpha}^2 + \sum_{j}\dim_{K}(\Flag_{\bn_{j}^{(\alpha)}}(\sE_{\alpha}|_{p_{j}}))    
	$$
	
	for some integers $r_{i\alpha} \geq 1$ such that $\sum_{i} r_{i\alpha} = r_{\alpha}$. 	
	We also have that 
	$$\sum_{\alpha} \sum_{j} \dim_{K}\Flag_{\bn_{j}^{(\alpha)}}(\sE_{\alpha}|_{p_{j}}) \leq \sum_{j} \dim_{K}\Flag_{\bn_{j}}(\sE|_{p_{j}}),$$ see \cite[11.1]{nicole}.	
	Using $$
	\trdeg_{k}k(\sE) \leq \sum_{j}\trdeg_{k}k(\sE_{j})
	$$	
	we have 	
	$$
	\trdeg_{k}k(\sE) \leq \sum_{j}1 + (g-1) \sum_{i,\alpha}r_{i\alpha}^2 + \sum_{j} \dim_{K}\Flag_{\bn_{j}}(\sE|_{p_{j}})
	$$	
Note that the sum $\sum_{i,\alpha}r_{i\alpha} = r$ has at least two terms, (cf \cite[6.5]{crelle}).	
	Hence 
	$$
	\trdeg_{k}k(\sE) \leq (g-1)(r^2-r) + \sum_{j} \dim_{K}\Flag_{\bn_{j}}(\sE|_{p_{j}}) + 2 - (g-2)(r-2)
	$$
	$r \geq 2$ and $g \geq 2$ implies $(g-2)(r-2) \geq 0$. 	
	So, $$
	\trdeg_{k}k(\sE) \leq (g-1)(r^2 - r) + \sum_{j} \dim_{K}\Flag_{\bn_{j}}(\sE|_{p_{j}}) + 2
	$$
	
\end{proof}

\section{Essential dimension on curves}

In this setion $X$ is a smooth projective, geometrically connected curve over $k$ with
$X(k)\ne\varnothing$.
 We fix some closed points $p_1,\ldots, p_l$ with $\deg p_i=f_i$. We will consider vector
 bundles of rank $r$ and degree $d$ and  ramification indices $e_i$ at each $p_i$.
We fix parabolic data $[r=n_{i0}\ge n_{i1} \ge \ldots \ge n_{ie_i}=0]=\bn_i$ at each $p_i$.
Let $\bn = ((p_1,\bn_1),\ldots, (p_l,\bn_l)$ be the corresponding parabolic datum.
The arguments in this section are modeled by those in \cite{crelle}.

A parabolic vector bundle $\sF$ on $X_K$ is said to be \emph{simple} if 
$\End(\sF)=K$. If $\sF$ is simple then the morphism from the residual gerbe to its moduli space 
$$ \sG(\sF)\rightarrow \Spec(k(\sF))$$
is banded by $\Gm$. As a generic parabolic bundle is simple this is in
fact the generic gerbe of the moduli stack. We wish to understand the index
of this gerbe.

\begin{proposition}\label{p:index} In the above situation,
	set  $h = \gcd(r,d,n_{ij})$. Then $h=\ind(\sG(\sF))$ where $\sG(\sF)$ is the generic gerbe.
\end{proposition}

\begin{proof}
 The index of this gerbe divides $h = \gcd(r,d,n_{ij})$.
 The proof will make use of twisted sheaves and we refer the reader
 to \cite{lieblich} for an introduction and their relationship to 
 the index. 
The universal parabolic bundle on 
\[
\Bun^{r,d}_{\bn}\times X
\]
produces twisted sheaves of ranks $r$ and $n_{ij}$ on this stack.

We restrict them  to the gerbe using the fact that $X(k)\ne\varnothing$.
By \cite{lieblich} we see that the index divides $r$ and the $n_{ij}$.
We have another twisted sheaf obtained by taking $\pi_*$ of a sufficiently ample
twist of the universal bundle. Its rank is computed by Riemann-Roch, see \ref{t:stackRR}, and
hence the index divides $h$.

Now choose a particular simple parabolic bundle $\sF_0$. 
We can consider the moduli stack of $\Bun^{r,d}_{\bn, \det(\sF_0)}$  of parabolic bundles where the underlying bundle has determinant $\det(\sF)$. The stack $\Bun^{r,d}_{\bn}$ is a Grassman bundle over the moduli stack of ordinary vector bundles and hence
by
\cite[Theorem 6.1]{norbert} and \cite{gabber} the generic gerbe of
$\Bun^{r,d}_{\bn, \det(\sF_0)}$ has index $h$. As the index can only drop by
base change, it follows that our original gerbe had index $h$.
In the case of a ground field of characteristic 0, one could apply the main theorem of \cite{dey}.
\end{proof}

\begin{proposition}
	Let $\sF$ be a simple parabolic vector bundle with rank $r$, degree $d$ and
	specified parabolic data. Then
	$$ \ed_k \sF\le  r^2(g-1)+1+\sum_{i=1}^l f_i\dim_{k(p_i)} \Flag_{\bn_i}
	+ \sum_{p|h} p^{v_p(h)}-1 $$
	and
	$$
	\ed_{k,p}\sF =   r^2(g-1)+1+\sum_{i=1}^l f_i\dim_{k(p_i)} \Flag_{\bn_i}
	+  p^{v_p(h)}-1.
	$$
\end{proposition}

\begin{proof}
	One combines the above proposition with  \ref{trdeg} and \ref{e:gerbe}.
\end{proof}

\begin{theorem}\label{t:main}
	Set $h = \gcd(r,d,n_{ij})$. We have
\[
\ed Bun^{r,d}_{\bn} \le  r^2(g-1)+1+\sum_{i=1}^l f_i\dim_{k(p_i)} \Flag_{\bn_i}
+ \sum_{p|h} p^{v_p(h)}-1.\]
Further,
$$
\ed_{p} Bun^{r,d}_{\bn} =   r^2(g-1)+1+\sum_{i=1}^l f_i\dim_{k(p_i)} \Flag_{\bn_i}
+  p^{v_p(h)}-1.
$$
When the main conjecture of \cite{ct} holds for $r$ then the first inequality is an 
equality.

\end{theorem}

\begin{proof}
	Using \ref{p:rank} and its proof that 
	\[ \ed_{k(\sF)} \sF \le r-1 \qquad \ed_{k(\sF),p} \sF \le v_p(r)-1 \]
	for every parabolic bundle.
	
	In the case where $\sF$ is not simple we can combine this remark with \ref{trdeg2}
	to obtain the result inequalities in the assertions of the theorem.
	
	The case of a simple bundle is the prior proposition.
	
	The conjecture of \cite{ct} relates the essential dimension of a gerbe
	to its index so that the equality is a consequence of \ref{p:index}.
\end{proof}

\end{document}